\documentclass[a4paper]{article}

\usepackage{a4wide}
\usepackage{amsmath,amssymb,amsthm}
\usepackage{mathrsfs,bbm}
\usepackage[american]{babel}
\usepackage[dvips]{graphicx}
\usepackage{array,pifont,varioref,framed}

\graphicspath{{./}{figure/}}

\title{A Nystr\"{o}m method for a boundary integral equation related to the Dirichlet problem on domains with corners}
\author{Luisa Fermo \\
		{\small\it Department of Mathematics and Computer Science} \\
		{\small\it University of Cagliari} \\
		{\small\it Viale Merello 92, 09123 Cagliari, Italy} \\[  5mm]
		Concetta Laurita \\
		{\small\it Department of Mathematics, Computer Science and Economics} \\
		{\small\it University of Basilicata} \\		
		{\small\it Via dell'Ateneo Lucano 10, 85100 Potenza, Italy}
	   }
\date{}

\newcommand{\C}{\mathcal{C}}

\newcommand{\RR}{\mathbb{R}}
\newcommand{\NN}{\mathbb{N}}
\newcommand{\PP}{\mathbb{P}}

\def\C{{\mathcal{C}}}

\theoremstyle{plain}\newtheorem{theorem}{Theorem}[section]
\theoremstyle{remark}
\theoremstyle{plain}\newtheorem{lemma}[theorem]{Lemma}
\theoremstyle{plain}\newtheorem{corollary}[theorem]{Corollary}
\theoremstyle{definition}

\begin{document}
\maketitle

\begin{abstract}
The authors consider the interior Dirichlet problem for Laplace's equation on planar domains with corners. In order to approximate the solution of the corresponding double layer boundary integral equation, they propose a numerical method of Nystr\"{o}m type, based on a Lobatto quadrature rule.
The convergence and stability of the method are proved and some numerical tests are included.

\medskip


\medskip

\end{abstract}

\section{Introduction}
Let $D$ be a simply connected bounded region in the plane and let its boundary $\Sigma$ be a simple closed piecewise smooth curve.
Let us assume $\Sigma$ at least twice continuously differentiable, with the exception of corners at some points $P_1,\ldots,P_n$.

We consider the interior Dirichlet problem for Laplace's equation
\begin{equation}\label{Dirichlet}
\left\{
  \begin{array}{ll}
    \Delta u(P)=0, & \hbox{$P \in D$} \\
    u(P)=g(P), & \hbox{$P \in \Sigma$}
  \end{array}
\right.
\end{equation}
where  $g$ is a given sufficiently smooth boundary function on $\Sigma$.

Using a double layer potential representation for the solution of (\ref{Dirichlet})
\begin{equation}\label{double}
u(A)=\int_{\Sigma} \psi(Q) \frac{\partial}{\partial \mathbf{n}_Q}[\log{|A-Q|}] d \Sigma_{Q}, \quad A \in D,
\end{equation}
with $\mathbf{n}_Q$ the inner normal to $\Sigma$ at $Q$, leads to the  BIE of the second kind (see, for instance, \cite{A})
\begin{equation}\label{equ}
(-2\pi+\Omega(P)) \psi(P) +\int_\Sigma \psi(Q) \frac{\partial}{\partial\mathbf{n}_Q} [\log{|P-Q|}] d \Sigma_{Q}= g(P), \quad P \in \Sigma,
\end{equation}
whose  unknown is  the so-called double layer density function $\psi(P)$  and where
$\Omega(P)$ denotes the interior angle to $\Sigma$ at $P$.
Note that $\Omega(P)=\pi$ if $\Sigma$ is smooth in $P$, while in the ``corner points'' we assume
\[0<\Omega(P)<2\pi.\]

Defining the operator
\begin{equation} \label{opK}
(K \psi)(P)=(-\pi+\Omega(P))\psi(P)+\int_\Sigma \psi(Q) \frac{\partial}{\partial\mathbf{n}_Q} [\log{|P-Q|}] d \Sigma_{Q}, \quad P \in \Sigma
\end{equation}
for $\psi \in C(\Sigma)$, one can rewrite equation (\ref{equ}) in the following more compact operator form
\begin{equation} \label{equop}
(-\pi+K)\psi=g.
\end{equation}
It is well known (see, for instance,  \cite{A}) that the operator ${ K}$ is a bounded map from $C(\Sigma)$ into $C(\Sigma)$ and is compact
when $\Sigma$ is a smooth curve. On the other hand, $K$ is no longer compact when the boundary $\Sigma$ is only piecewise smooth, due to the presence of
the corner points. In addition, the double layer density function may have a singularity in corners of the type
\[d^{\beta}, \quad \beta=\frac{\pi}{\pi+|\pi -\phi|},\]
with $d$  the distance from the corner and $\phi$ the interior angle at the corner.

The  most po\-pu\-lar methods to solve such a problem, for instance collocation, Galerkin and Nystr\"om methods, are based on piecewise polynomial approximations with graded meshes
(see, for example, \cite{Chandler,Jeon,Kress,Rathsfeld} and the references therein).
The use of this type of approximation allows, by grading properly the mesh, one  to obtain arbitrarily high orders of convergence.
Nevertheless, the final linear systems one has to solve become ill-conditioned as the local degree increases.

A different approach is proposed in \cite{MS} where the authors describe a method based on a global ap\-pro\-xi\-mation of the unknown function.
By re\-pre\-senting the solution of  the Dirichlet problem in the form of a single layer potential, they reduce it to solving a system of integral
equations of the first kind. Then, after introducing some smoothing changes of variable, they apply a collocation method  ap\-pro\-xi\-mating the
unknown density by means of polynomials over each smooth section of the boundary.
The numerical results  improve as the regularizing  parameter increases.
Unfortunately, in \cite{MS}, the stability and convergence of the described numerical procedure are not theoretically proved and error estimates are not given.

More recently an extensive literature on efficient numerical methods to discretize boundary integral equations connected with elliptic problems
on domains with corners has been developed (see \cite{Bremer2012_1,Bremer2012_2,BremerRokhlin,BremerRokhlinSammis,BrunoOvalTurc,Helsing2011,HelsingOjala}
and the references therein). \newline
In \cite{BrunoOvalTurc} a  scheme for the numerical solution of the Neumann problem for the Laplace equation is introduced.
The solutions of the standard corresponding integral equations can be unbounded at the corners. In order to achieve high accuracy in the computation
 of the  Nystr\"om solution, the authors propose the analytical subtraction of singularities and a special treatment of nearly non-integrable
integrands in such a way as to avoid cancellation errors. \newline
 A scheme dubbed ``recursive inverse preconditioning'' has been introduced in \cite{HelsingOjala} and further developed in \cite{Helsing2011}. It is a technique which allows to overcome the negative effects of
the ill-conditioning of matrices arising from Nystr\"om discretization of singular integral equations on non-smooth domains. \newline
In \cite{Bremer2012_1,Bremer2012_2} the author presents a Nystr\"om method based on discretization techniques described in  \cite{BremerRokhlin}
and \cite{BremerRokhlinSammis}. The advantage of this method is that, in addition to pro\-du\-cing well conditioned linear systems,
thanks to a compression scheme, the approach reduces the number of equations which becomes excessively large in the presence of large-scale domains
with corners. Nevertheless, in these papers the mathematical analysis of stability and convergence of the proposed procedures is not carried out but are
only demonstrated through several numerical examples which show high computational accuracy.

Here we propose a numerical method, based on global approximations, that directly produces  well conditioned systems without resorting to
preconditioning techniques.
By following already known ideas (see \cite{A} and the references therein), we  decompose, in a suitable way, the piecewise smooth boundary $\Sigma$
into sections and convert the boundary integral equation (\ref{equop}) into an equivalent system of integral equations of the second kind.
Then a Nystr\"om  method using a global ap\-pro\-xi\-mation over each smooth section of the boundary is applied for its numerical solution.
The method applies the Lobatto quadrature rule in order to evaluate the integrals involved in the system.
In any case, a slight modification of the corresponding discrete operator around the corners is needed in the approximating system to achieve stability.
Finally, the solution of the linear system to which the method leads is used to calculate a discrete approximation of the  double layer potential
(\ref{double}).
We are able to prove theoretically the stability  and convergence of the proposed procedure.
 Moreover, we also show that
the linear systems  arising from the discretization of the system of boundary integral equations are well conditioned.
Neverthless, for domains with a large number of corner points the procedure involves high computational costs as the dimension of the linear system
increases.

The contents of the paper are as follows.
Section \ref{Section2} provides preliminary definitions, notations and results.
Section \ref{Section3} is devoted to describing the numerical procedure and to establishing the main theorems about its stability and convergence.
Section \ref{Section4} contains the proofs of the theoretical results and, finally, Section \ref{Section5} concludes the paper
by presenting some numerical tests.

\section{Preliminaries} \label{Section2}
\subsection{Function spaces}
Let  $L^p([0,1])$ be the space of all measurable functions $f$ on $[0,1]$ such that
\[\|f\|_p=\left(\int_{0}^{1} |f(x)|^p dx \right)^{\frac 1 p}< +\infty, \quad 1 \leq p <+\infty. \]
With $w(x)\in L^p([0,1])$, a Jacobi weight on $[0,1]$, we set $f \in L_w^p([0,1])$ if and only if $fw \in L^p([0,1])$, $1 \leq p <+\infty$
and equip the space $L_w^p([0,1])$ with the norm
\[\|fw\|_p=\left(\int_{0}^{1} |f(x)w(x)|^p dx \right)^{\frac 1 p}, \quad 1 \leq p <+\infty.\]
Moreover, we consider the  Sobolev-type subspace $W^p_r(w)$ of $L_w^p([0,1])$ defined as follows
\[W^p_r(w)=\{f \in L_w^p([0,1]) \ \mid \ \|f\|_{W^p_r(w)}= \|fw\|_p +\|f^{(r)}\varphi^rw\|_p<\infty\},\]
where $r$ is a positive integer and $\varphi(x)= \sqrt{x(1-x)}$.

Finally, for any $r \in \NN_{0} \cup \{\infty\}$, let us consider the following direct product
\[C^r([0,1])^{^m}=\underbrace{C^r([0,1]) \times C^r([0,1]) \times \ldots \times C^r([0,1])}_m,\]
which is complete equipped with the norm
\begin{equation} \label{normdirectproduct}
\|(f_1,f_2,\ldots,f_m)\|_{\infty} =\max_{i=1,2,\ldots,m} \|f_i\|_\infty.
\end{equation}

\subsection{The Lobatto quadrature rule}
In this subsection we give some remarks on the well-known Lobatto quadrature rule (see, for instance, \cite[p. 104]{DR}), since we are
going to use a method of Nystr\"{o}m type,  based on this integration formula, for the numerical solution of our system of integral equations.
Let us premise some notations.

In the sequel $\C$ denotes a positive constant which may assume different values in different formulas. We write $\C =\C(a,b,...)$ to say that
 $\C$ is dependent on the parameters $a,b,....$ and $\C \neq \C(a,b,...)$ to say that $\C$ is independent of them.
Furthermore, if $A,B > 0$ are quantities depending on some parameters, we will write $A \sim B$, if there exists a positive constant $\C$ independent of the parameters of $A$ and $B$ such that
\[ \frac{1}{\C} \leq \frac{A}{B} \leq \C.\]

Let $w\in L^p([0,1])$ be a Jacobi weight on $[0,1]$. Denoting by $\PP_m$ the set of all algebraic polynomials of degree at most $m$,
for functions $f\in L_w^p$,  we define the  weighted error of best polynomial approximation as
\[E_m(f)_{w,p}=\inf_{P_m \in \PP_m} \left\|\left(f-P_m\right)w\right\|_p.\]
Moreover, for simplicity, in the case when $w(x)\equiv 1$ we set $W^p_r=W^p_r(w)$ and $E_m(f)_{p}=E_m(f)_{w,p}$ . \newline
Now, let $\{p_m(v^{1,1})\}_m$, $p_m(v^{1,1}) \in \PP_m$, be the
sequence of  polynomials which are orthogonal on $[0,1]$ with respect to the Jacobi weight
$v^{1,1}(x)=x(1-x)$. The Lobatto quadrature rule  over the interval $[0,1]$ is given by
\begin{equation}\label{lobatto}
\int_0^1 f(x) dx= \sum_{k=0}^{m+1}  \lambda_k f(x_k) + e_m(f)
\end{equation}
with the nodes $x_0=0$,  $x_1<x_2<...<x_m$ zeros of $p_m(v^{1,1})$, $x_{m+1}=1$
and the coefficients
\begin{equation} \label{Lobattoweights1}
\lambda_0=\lambda_{m+1}=\frac{1}{(m+1)(m+2)},
\end{equation}
\begin{equation} \label{Lobattoweights2}
\lambda_k=\int_0^1\frac{l_k(x)v^{1,1}(x)}{v^{1,1}(x_k)}dx, \quad k=1,\ldots,m,
\end{equation}
where $l_k(x)$ is the $k$-th fundamental Lagrange polynomial based on the points $x_1,\ldots,x_m$.
Finally, $e_m(f)$ in (\ref{lobatto}) denotes the remainder term.

The following results give an error estimate for the Lobatto rule (\ref{lobatto}) in the case when $f \in W^1_r$, $r \geq 1$.
\begin{theorem}\label{errorequad}
For all $f \in W^1_1$ we have
\begin{equation} \label{errorequadest}
|e_m(f)| \leq \frac{\C}{m} E_{2m}(f')_{\varphi,1}
\end{equation}
where $\varphi(x)=\sqrt{x(1-x)}$ and $\C \neq \C(m,f)$.
\end{theorem}

From the previous theorem and taking into account the Favard inequality (see \cite{DT})
\begin{equation}\label{Favard}
E_m(f)_{w,p} \leq \frac \C {m^r}\, E_{m-r}(f^{(r)})_{\varphi^rw,p}, \quad  \C \neq \C(m,f),
\end{equation}
holding true for each function $f  \in W^p_r(w)$,
we can immediately deduce the following
\begin{corollary}\label{corquad}
For all $f \in W^1_r$, $r \geq 1$, we have
\begin{equation}\label{errorequad2}
|e_m(f)| \leq \frac{\C}{m^r} E_{2m+1-r}(f^{(r)})_{\varphi^r,1}
\end{equation}
where $\varphi(x)=\sqrt{x(1-x)}$ and $\C \neq \C(m,f)$.
\end{corollary}

\section{The method} \label{Section3}
In this section we are going to propose a method to approximate  the solution of the boundary integral equation (\ref{equop}).
In order to simplify the presentation,  we shall consider the case where the boundary $\Sigma$ has only one corner at a point $P_0$ with an interior
angle $\phi=(1-\chi) \pi$, $-1<\chi<1$, $\chi \neq 0$. The extension to boundary curves with more than one corner is straightforward.

As recalled in the introduction, in the case under consideration the operator $K$ in (\ref{opK}) is not compact, but
the following splitting of $K$ is possible (\cite{AdH,ChandlerGraham,Jeon})
\[
K=\hat{L}+M,
\]
with $M$ a compact operator from $C(\Sigma)$ into $C(\Sigma)$ and $\hat{L}$ essentially the so called ``wedge operator'', i.e. the operator
$K$ defined on the wedge having vertex at $P_0$ and arms tangent to those of the boundary $\Sigma$ in the neighborhood of the corner point.
The operator $\hat{L}$, which is not compact, satisfies $\|\hat{L}\|<\pi$. Hence, in the decomposition $-\pi+\hat{L}+M$ of  $-\pi+K$, the operator
$-\pi+\hat{L}$ has a bounded inverse by the Neumann series.
Therefore, if  $-\pi+\hat{L}+M$ is injective, the inverse operator $(-\pi+\hat{L}+M)^{-1}:C(\Sigma) \rightarrow C(\Sigma)$ exists and is bounded.
Nevertheless, we don't apply the Nystr\"om method directly to the initial double layer potential equation
\[(-\pi+\hat{L}+M)\psi=g,\]
since the kernel of the operator $M$  is bounded but could be discontinuous at the corner point (see \cite{AdH}) and this
would make more difficult the theoretical analysis of the stability and convergence of the numerical procedure.

The method we are going to propose consists of two basic steps. As a first step we decompose,  in a suitable way,  the curve into sections and reduce (\ref{equop}) to an e\-qui\-valent system of integral equations. The second step is to apply a numerical method of Nystr\"{o}m type, based on the Lobatto quadrature rule (\ref{lobatto}),
to compute the solution of such a system.

Begin by subdividing $\Sigma$ into the sections $\Sigma_1$, $\Sigma_2$ and $\Sigma_3$ defined as follows.
By proceeding in the counterclockwise direction, let $\Sigma_1$ and $\Sigma_2$ be  two
sufficiently small smooth arcs of the boundary $\Sigma$ intersecting at the corner $P_0$. Moreover, we assume that their lengths are chosen so that
$\Sigma_1$ and $\Sigma_2$
essentially coincide with the segments tangent to the curve $\Sigma$ at $P_0$, in the sense
that
\begin{equation} \label{delta}
\max_{(x,y)\in \Sigma_i}|y-y_{t_i}| \leq \delta, \quad i=1,2,
\end{equation}
where $y_{t_i}$ denotes the ordinate of the point with abscissa $x$ on the segment tangent to $\Sigma_i$ at $P_0$ and $\delta$ is a very small positive number.
Finally, let   $\Sigma_3$  be the section connecting $\Sigma_1$ and $\Sigma_2$.

Then we can rewrite the boundary integral equation (\ref{equ}) as the following system of $3$ boundary integral equations
\begin{eqnarray}
\label{sist}
(-2 \pi + \Omega(P)) \psi_i(P)+ \sum_{j=1}^{3} \int_{\Sigma_j} \psi_j(Q) \frac{\partial}{\partial \mathbf{n}_Q} [\log{|P-Q|}] d \Sigma_{Q}&=&g_i(P), \\ & \quad & P \in \Sigma_i, \quad i=1,2,3 \nonumber
\end{eqnarray}
where $\psi_i$ and $g_i$ denote the restrictions of the functions $\psi \in C(\Sigma)$ and $g \in C(\Sigma)$ to the curve $\Sigma_i$, respectively.

In order to transform the above curvilinear 2D integrals into 1D integrals on the same reference interval, let us introduce a parametric representation $\sigma_i$ defined on the interval $[0,1]$ for each arc $\Sigma_i$
\begin{equation}\label{trasf}
\sigma_i: s \in [0,1] \to (\xi_i(s), \eta_i(s)) \in \Sigma_i,
\end{equation}
with $\sigma_i \in C^{2}([0,1])$ and $|\sigma_i'(s)|\neq 0$ for each $0 \leq s \leq 1$ and $i=1,2,3$. Moreover,
without any loss of generality, we assume that $\sigma_1'(s)<0$, $\sigma_2'(s)>0$, $\sigma_3'(s)>0$, $\forall s \in [0,1]$, and $\sigma_1(0)=\sigma_2(0)=P_0$.
Then (\ref{sist}) can be rewritten as the following system of integral equations on the interval $[0,1]$
\begin{equation}  \label{sist2}
(-2 \pi +\bar{\Omega}_i(s))\bar{\psi}_i(s)+ \sum_{j=1}^{3} \int_0^1K^{i,j}(t,s)\bar{\psi}_j(t)dt= \bar{g}_i(s), \quad s\in[0,1], \quad i=1,2,3,
\end{equation}
where $\bar{\Omega}_i(s)=\Omega(\sigma_i(s))$, $\bar{\psi}_i(s)=\psi_i(\sigma_i(s))$, $\bar{g}_i(s)=g_i(\sigma_i(s))$
and
\begin{equation} \nonumber
K^{i,j}(t,s)=\left\{
    \begin{array}{ll}
      \displaystyle \frac{\eta'_j(t)[\xi_i(s)-\xi_j(t)]-\xi'_j(t)[\eta_i(s)-\eta_j(t)]}{[\xi_i(s)-\xi_j(t)]^2+[\eta_i(s)-\eta_j(t)]^2}, & \hbox{$i \neq j \quad \mathrm{or} \quad t \neq s$} \\ \\
    \displaystyle \frac 1 2 \frac{\eta'_j(t) \xi''_j(t)-\xi_j'(t) \eta''_j(t)}{[\xi'_j(t)]^2+[\eta'_j(t)]^2}, & \hbox{$i=j \quad \mathrm{and} \quad t=s$}
    \end{array}
  \right.
\end{equation}
for $t,s \in [0,1]$. Note that
\begin{equation} \label{angle}
\bar{\Omega}_1(s)=\bar{\Omega}_2(s)=\left\{
\begin{array}{ll}
\pi,  & \quad 0<s \leq 1\vspace{0.2cm}\\
(1-\chi)\pi,  & \quad s=0
\end{array}
\right.
\end{equation}
and $\bar{\Omega}_3(s)=\pi$ for all $s \in [0,1]$.

Now, \ let us introduce the following complete subspace of the product space $C([0,1])^{^3}$ equipped with the norm (\ref{normdirectproduct}),
$$\tilde{\mathcal X}=\left\{(f_1,f_2,f_3)^T \in C([0,1])^{^3}  \mid f_1(0)=f_{2}(0), \ f_2(1)=f_3(0), \ f_1(1)=f_3(1)\right\}$$
and the bijective map $\eta:C(\Sigma)\rightarrow \tilde{\mathcal X}$ defined as follows
$$\eta f=(\bar{f}_1,\bar{f}_2,\bar{f}_3), \quad \bar{f}_i(t)=f(\sigma_i(t)), \quad t \in [0,1].$$
By defining the following matrices of operators
\begin{equation}
{\mathcal I} =\left(
                           \begin{array}{lll}
                             I & 0 &  0 \\
                              0& I  & 0   \\
                             0 &   0 & I \\
                           \end{array}
                         \right), \quad \quad
{\mathcal K} =\left(
                           \begin{array}{lll}
                              (-\pi+\bar{\Omega}_1)I+\mathcal K^{1,1} &  \mathcal K^{1,2} &   \mathcal K^{1,3}\\
                              \mathcal K^{2,1}   & (-\pi+\bar{\Omega}_2)I+\mathcal K^{2,2} & \mathcal K^{2,3}\\
                              \mathcal K^{3,1} & \mathcal K^{3,2} & \mathcal K^{3,3}\\
                           \end{array}
                         \right),
\end{equation}
with $I$ the identity operator on the space $C([0,1])$ and
\[
(\mathcal K^{i,j} \rho)(s)=\int_{0}^1 K^{i,j}(t,s) \rho(t) dt, \quad \rho \in C([0,1]),
\]
the system (\ref{sist2}) can be rewritten, in a compact form, as follows
\begin{equation} \label{sistop}
(-\pi {\mathcal I}+ {\mathcal K}) \bar{\psi}=\bar{g},
\end{equation}
where
\begin{equation} \label{barpsi}
\bar{\psi}= \left(\bar{\psi}_1,\bar{\psi}_2,\bar{\psi}_3\right)^T \in \tilde{\mathcal X}, \quad \bar{g}= \left(\bar{g}_1,\bar{g}_2,\bar{g}_3\right)^T
\in \tilde{\mathcal X}.
\end{equation}
Let us observe that the operator $(-\pi {\mathcal I}+ {\mathcal K})^{-1}: \tilde{\mathcal X} \rightarrow \tilde{\mathcal X}$ exists and is bounded
since we have
\begin{equation} \label{opeta}
(-\pi {\mathcal I}+ {\mathcal K})=\eta(-\pi+K)\eta^{-1}.
\end{equation}
Moreover, let us note that the integral operators $\mathcal K^{i,j}$
are compact on the space $C([0,1])$, 
since their kernels are continuous on $[0,1]\times[0,1]$
(see, for instance, \cite{AdH,Kress}), except when $i,j\in\{1,2\}$  and $i \neq j$.
In fact, in such cases $\mathcal K^{i,j}$ takes the following form (see \cite{AdH,ChandlerGraham,Jeon})
\begin{equation}\label{Mellin}
(\mathcal K^{i,j} \rho)(s)=(\mathcal L \rho)(s)+(\mathcal M^{i,j}\rho)(s),
\end{equation}
where the integral operator $\mathcal L $ is defined as follows
\[
(\mathcal L \rho)(s)=\int_0^1 L (t,s) \rho(t) dt, \quad \rho \in C([0,1]),
\]
with the  Mellin--type kernel $L(t,s)$ given by
\[
L(t,s)=-\frac{s \sin {(\chi \pi)}}{s^2+2ts \cos {(\chi \pi)}+t^2},
\]
and
\[
(\mathcal M^{i,j}\rho)(s)=\int_0^1 M^{i,j}(t,s)\rho(t)dt, \quad \rho \in C([0,1]),
\]
with the kernel $ M^{i,j}(t,s)$ continuous on $[0,1]\times[0,1]$.



In order to carry out the theoretical analysis of
the stability and convergence of the numerical procedure we are going to
propose,  we rewrite (\ref{sistop}) as follows
\begin{equation}\label{sistop1}
(-\pi \mathcal I+\mathcal W+\mathcal S)\bar{\psi}=\bar{g}
\end{equation}
with
\begin{equation} \label{MatrixW}
{\mathcal W}=\left(
                           \begin{array}{lll}
                              (-\pi+\bar{\Omega}_1)I &  \mathcal L &   0\\
                               \mathcal  L   & (-\pi+\bar{\Omega}_2)I & 0\\
                              0 & 0 & 0\\
                           \end{array}
                         \right)
\end{equation}
and
\begin{equation} \label{MatrixS}
{\mathcal S}={\mathcal K}-{\mathcal W}
\end{equation}
and we introduce the following complete subspace of the product space $C([0,1])^{^3}$
\begin{equation} \label{space}
{\mathcal X}=\left\{(f_1,f_2,f_3)^T \in C([0,1])^{^3}  \mid f_1(0)=f_{2}(0)\right\}.
\end{equation}
Note that $\tilde{\mathcal X} \subset {\mathcal X}$.
Then we are able to prove the following result concerning the solvability of the system (\ref{sistop1}) in the spaces ${\mathcal X}$ and $\tilde{\mathcal X}$.

\begin{theorem}\label{theoremsolv}
Let $\mathrm{Ker}(-\pi \mathcal{I}+{\mathcal W}+{\mathcal S})=\{0\}$ in the Banach space ${\mathcal X}$. Then system $(\ref{sistop1})$ has a unique solution in ${\mathcal X}$ for each given right hand side $\bar{g} \in {\mathcal X}$. Moreover, if $\bar{g} \in \tilde{\mathcal X}$ then the solution $\bar{\psi}$ of $(\ref{sistop1})$ also belongs to $\tilde{\mathcal X}$.
\end{theorem} \vspace{0.2cm}

Moreover, it is known that (see \cite{A,Chandler,Gris} and the re\-fe\-ren\-ces therein) even if the Dirichlet data $g$ is a smooth function,  the solution $\bar{\psi}=\left(\bar{\psi}_1,\bar{\psi}_2,\bar{\psi}_3\right)^T$ of (\ref{sistop1}) satisfies the following
smoothness properties:

\begin{itemize}
\item $\bar{\psi}_3$ is smooth; 
\item for $i \in \{1,2\}$
\begin{equation} \label{behavioursolution}
\bar{\psi}_i(t)=O\left(t^{\beta}\right), \quad  0 <t \leq 1, \quad \beta=\frac{1}{1+|\chi|},
\end{equation}
\begin{equation} \label{behavioursolution2}
\bar{\psi}_i^{(r)}(t) \leq \C t^{\beta-r}, \quad  0 <t \leq 1, \quad r=1,2,\ldots \quad.
\end{equation}
\end{itemize}
Therefore, there will almost always be an algebraic singularity in the first derivative of the double layer density function $\psi$ near the corner points, being $\frac{1}{2}<\beta<1$.

Now, in order to approximate  the solution of (\ref{sistop}) or, equivalently, of (\ref{sistop1}), we are going to  propose a numerical method of Nystr\"{o}m type  based on the Lobatto
quadrature rule (\ref{lobatto}). \newline
Then, for any fixed $m \in \NN$, denoting by $\lambda_h$ and $x_h$, $h=0,1,\ldots,m+1$, the coefficients and the nodes of  formula (\ref{lobatto}), respectively, we define the following finite rank operators
\begin{equation} \label{Lm}
(\mathcal L_m \rho)(s)= \sum_{h=0}^{m+1} \lambda_h L(x_h,s) \rho(x_h),
\end{equation}
approximating $\mathcal L$,  and
\begin{equation} \label{Mmij}
(\mathcal M_m^{i,j} \rho)(s)= \sum_{h=0}^{m+1} \lambda_h  M^{i,j}(x_h,s) \rho(x_h),
\end{equation}
\begin{equation} \label{Kmij}
(\mathcal K_m^{i,j} \rho)(s)= \sum_{h=0}^{m+1} \lambda_h  K^{i,j}(x_h,s) \rho(x_h),
\end{equation}
approximating the entries $\mathcal M^{i,j}$ and $\mathcal K^{i,j}$  of the matrix $\mathcal S$, respectively.

Now, any sequence of operators $\left\{\mathcal K_m^{i,j}\right\}_m$ is pointwise convergent to the operator $\mathcal K^{i,j}$ in the space $C([0,1])$, as well as $\left\{\mathcal M_m^{i,j}\rho\right\}_m$ tends to $\mathcal M^{i,j}\rho$ for any continuous function $\rho$ on $[0,1]$. On the other hand, it is possible to prove that, for a function $\rho \in C([0,1])$,  the sequence of functions $\left\{ \mathcal L_m\rho\right\}_m$ converges uniformly to $\mathcal L\rho$ in any interval of the type $\left[\displaystyle \frac{c}{m^{2-2\epsilon}},1\right]$, for some constant $c>0$ and arbitrarily small $\epsilon>0$ (see Lemma \ref{propLm}) and does not converge in $[0,1]$.

Let us introduce the following matrices of operators
\begin{equation} \label{Wm}
{\mathcal W_m}=\left(
                           \begin{array}{lll}
                              (-\pi+\bar{\Omega}_1)I &  \mathcal L_m &   0\\
                               \mathcal  L_m   & (-\pi+\bar{\Omega}_2)I & 0\\
                              0 & 0 & 0\\
                           \end{array}
                         \right)
\end{equation}
and
\begin{equation} \label{MatrixSm}
 {\mathcal S}_m =\left(
                           \begin{array}{lll}
                              \mathcal K_m^{1,1} &  \mathcal M_m^{1,2} &   \mathcal K_m^{1,3}\\
                              \mathcal M_m^{1,2}   & \mathcal K_m^{2,2} & \mathcal K_m^{2,3}\\
                              \mathcal K_m^{3,1} & \mathcal K_m^{3,2} & \mathcal K_m^{3,3}\\
                           \end{array}
                         \right).
\end{equation}

In order to establish stability and convergence results for the procedure we are going to propose, fol\-lo\-wing an idea in \cite{MasMon},
we need to slightly modify just ${\mathcal W}_m$.
More precisely, for a fixed a constant $c>0$ and an arbitrarily small $\epsilon>0$,   we define

\begin{equation} \label{opmodified}
(\tilde{\mathcal W}_m \varrho)(s)=\left\{
  \begin{array}{ll}
    ({\mathcal W}_m \varrho)(s), & \hbox{$\displaystyle  \frac{c}{m^{2-2\epsilon}} \leq s \leq 1$} \\ \\
    \displaystyle \frac{m^{2-2\epsilon}}{c}\left[s({\mathcal W}_m \varrho)\left(\frac{c}{m^{2-2\epsilon}}\right)
    +\left(\frac{c}{m^{2-2\epsilon}} -s\right)({\mathcal W} \varrho)(0)\right],
   & \hbox{$0 \leq s < \displaystyle  \frac{c}{m^{2-2\epsilon}}$}
  \end{array} \right.
\end{equation}
with $\varrho=(\varrho_1,\varrho_2,\varrho_3)^T\in C([0,1])^{^3}$. \newline
\newline
The operators $\tilde{\mathcal W}_m$ and ${\mathcal S}_m$ satisfy the following theorems.
\begin{theorem} \label{lemmaWm}
Let ${\mathcal W}$ and $\tilde{\mathcal W}_m$ be defined in $(\ref{MatrixW})$ and $(\ref{opmodified})$, respectively.
Then the operators $\tilde{\mathcal W}_m:{\mathcal X} \to {\mathcal X}$ are linear maps such that
\begin{equation} \label{limnormWm}
 \lim_{m \to \infty} \|\tilde{\mathcal W}_m \| < \pi
\end{equation}
 and
 \begin{equation} \label{convpointnormWm}
 \lim_{m \to \infty} \|(\tilde{\mathcal W}_m -{\mathcal W})\rho\|_{\infty}=0, \quad \forall \ \rho \in {\mathcal X}.
 \end{equation}
\end{theorem}

\begin{theorem}\label{opercompatti}
Let ${\mathcal S}$ and ${\mathcal S_m}$ be defined in $(\ref{MatrixS})$ and $(\ref{MatrixSm})$, respectively.
Then the operators ${\mathcal S}_m:{\mathcal X} \to {\mathcal X}$ are linear maps such that the set $\left\{{\mathcal S_m}\right\}_m$ is collectively compact and
\begin{equation} \label{convpointnormSm}
\displaystyle \lim_{m \to \infty} \|({\mathcal S}_m -{\mathcal S}){\rho}\|_{\infty}=0, \quad \forall \  \rho \in {\mathcal X}.
 \end{equation}
\end{theorem}
The method we are proposing here consists of solving, instead of the  system of integral equations (\ref{sistop1}), the approximating one
\begin{equation} \label{sistapprox}
(-\pi \mathcal{I}+\tilde{\mathcal W}_m+{\mathcal S}_m)\bar{\psi}_m=\bar{g},
\end{equation}
whose unknown is the array of functions denoted by $\bar{\psi}_m=(\bar{\psi}_{m,1},\bar{\psi}_{m,2},\bar{\psi}_{m,3})^T$. \newline
In order to compute the solution $\bar{\psi}_m$ of (\ref{sistapprox}) at the quadrature nodes  $x_l$, $l=0,1,\ldots,m+1$, let us collocate each equation in these points. In this way we obtain the following linear system of $3(m+2)$ equations in the $3(m+2)$ unknowns $\bar{\psi}_{m,j}(x_l)$, $j=1,2,3$, $l=0,1,...,m+1$
\begin{equation} \label{linearsystem}
(-\pi \mathcal{I}+\tilde{\mathcal W}_m+{\mathcal S}_m)\bar{\psi}_m(x_l)=\bar{g}(x_l), \quad  \quad l=0,1,\ldots,m+1.
\end{equation}

Rewriting this linear system in the more compact  form
\begin{equation} \label{linearsystem3}
A_m\mathbf{a}=\mathbf{b},
\end{equation}
with $A_m$  the matrix of the coefficients,
$$\mathbf{a}=\left(a_{1,0},\ldots,a_{1,m+1},a_{2,0},\ldots,a_{2,m+1},a_{3,0},\ldots,a_{3,m+1}\right)$$ the array of the unknowns and
$$\mathbf{b}=\left(b_{1,0},\ldots,b_{1,m+1},b_{2,0},\ldots,b_{2,m+1},b_{3,0},\ldots,b_{3,m+1}\right)$$ the right hand side vector, we see that system (\ref{linearsystem3}) is equivalent to the approximating problem (\ref{sistapprox}) (see, for instance, \cite[p.101]{A}). More precisely, if
$\tilde{\RR}^{3(m+2)}$   denotes the subspace of $\RR^{3(m+2)}$ containing all the arrays
$$\left(c_{1,0},\ldots,c_{1,m+1},c_{2,0},\ldots,c_{2,m+1},c_{3,0},\ldots,c_{3,m+1}\right)$$
such that $c_{1,0}=c_{2,0}$, we have that each solution $\bar{\psi}_m \in {\mathcal X}$ of (\ref{sistapprox}) furnishes a solution $\mathbf{a}$ of system (\ref{linearsystem3}) belonging to $\tilde{\RR}^{3(m+2)}$. It will merely be sufficient to evaluate $\bar{\psi}_m(s)$ at the nodes of the Lobatto formula.
Viceversa, if $\mathbf{a} \in \tilde{\RR}^{3(m+2)}$ is a solution of (\ref{linearsystem3}), there is a unique $\bar{\psi}_m(s) \in {\mathcal X}$ which is
solution of (\ref{sistapprox}) such that
\begin{equation} \label{equivalsyst}
\bar{\psi}_{m,i}(x_l)=a_{i,l}, \quad i=1,2,3, \quad l=0,1,\ldots,m+1.
\end{equation}
Then we can conclude that the operator $-\pi \mathcal{I}+\tilde{\mathcal W}_m+{\mathcal S}_m$ is invertible on the space $\mathcal X$
if and only if the matrix $A_m$ is invertible on $\tilde{\RR}^{3(m+2)}$.

Before establishing our main result, let us make some remarks.\newline
 The first one concerns the computation of the entries of $A_m$. Note that, in order to construct this matrix,
one has to calculate the quantities  $\lambda_h L(x_h,x_l)$. The worst case could occur in the evaluation of
 $\lambda_0 L(x_0,x_1)$ because the values $L(x_0,x_1)$ increase more and more as well as $m$ increases.
Nevertheless, since $\lambda_0 \sim \displaystyle \frac 1{m^2}$ (see (\ref{Lobattoweights1})) and, using $x_1-x_0 \sim \displaystyle  \frac 1{m^2}$
(see (\ref{nodes})), it is easily seen that $L(x_0,x_1) \sim m^2$ (with the constants in $\sim$ independent of $m$), one has that the products
  $\lambda_0 L(x_0,x_1)$ are uniformly bounded with respect to $m$. \newline
As a second remark we would like to point out that,  for the sake of simplicity, we have used the same number $m+2$ of quadrature nodes for the
Lobatto formula in (\ref{Lm})-(\ref{Mmij}) and (\ref{Kmij}). Nevertheless, one can generalize the proposed procedure by using also different numbers of quadrature knots on each smooth arc of the boundary $\Sigma_j$, $j=1,2,3$.

\begin{theorem} \label{maintheorem}
Let $\Sigma\setminus\{P_0\}$ of class $C^{2}$.
Assume that $\mathrm{Ker}\{-\pi \mathcal{I}+ {\mathcal W}+ {\mathcal S}\}=\{0\}$ in the space ${\mathcal X}$. Then, for  sufficiently large $m$, say $m \geq m_0$,
the operators $-\pi \mathcal{I}+\tilde{\mathcal W}_m+{\mathcal S}_m$ are invertible and their inverses are uniformly bounded on ${\mathcal X}$.
Moreover, for all $\bar{g} \in {\mathcal X} \cap C^{p}([0,1])^{^{3}}$ with  $p$ large enough, the solutions $\bar{\psi}$ of equation $(\ref{sistop1})$ and $\bar{\psi}_m$ of $(\ref{sistapprox})$, satisfy the following error estimate
\begin{equation}\label{errorestimate}
\|(\bar{{\psi}}-\bar{{\psi}}_m)(s)\|_{\infty} \leq \C [\|(\tilde{\mathcal W}_m-{\mathcal W})\bar{{\psi}}(s)\|_{\infty}+
\|({{\mathcal S}}_m-{\mathcal S})\bar{{\psi}}(s)\|_{\infty}], \quad \C \neq \C(m),
\end{equation}
where
\begin{equation} \label{errorestimate1}
\|(\tilde{\mathcal W}_m-{\mathcal W})\bar{{\psi}}(s)\|_{\infty} \leq \left\{ \begin{array}{ll}
\C \max\left\{\left(\displaystyle \frac{1}{m^{2-2\epsilon}}\right)^{\beta},\displaystyle \frac{1}{m^{1+\epsilon}}\right\}, & s  \in \left[0,\displaystyle \frac {c}{m^{2-2\epsilon}}\right]\\
\displaystyle \frac{\C}{m^2}\frac{1}{s^{\frac{1}{2}}}, & s  \in \left[\displaystyle \frac {c}{m^{2-2\epsilon}},1\right]
\end{array} \right.,
\end{equation}
with $\epsilon$ as in  $(\ref{opmodified})$ and $\displaystyle \beta=\frac{1}{1+|\chi|}$.
\end{theorem}

Let us remark that (see Theorem \ref{opercompatti})
$$\lim_{m}\|({{\mathcal S}}_m-{\mathcal S})\bar{{\psi}}(s)\|_{\infty}=0, \quad \forall s\in [0,1]$$
and the rate of convergence depends on the smoothness of the boundary $\Sigma\setminus\{P_0\}$ as well as on the behavior of the functions
 $\bar{\psi}_j$ on the interval $[0,1]$ (see (\ref{behavioursolution}),
(\ref{behavioursolution2})). \newline

Moreover, we can prove the following theorem.
\begin{theorem} \label{theoremcondition}
Denoting by $\mathrm{cond}(-\pi \mathcal{I}+\tilde{\mathcal W}_m+{\mathcal S}_m)$ the condition number of the ope\-ra\-tor
$-\pi \mathcal{I}+\tilde{\mathcal W}_m+{\mathcal S}_m:{\mathcal X} \rightarrow {\mathcal X}$  and by $\mathrm{cond}(A_m)$
the condition number of the matrix $A_m:\tilde{\RR}^{3(m+2)} \rightarrow \tilde{\RR}^{3(m+2)}$  in infinity norm, we have that, for any $m \geq m_0$,
\begin{equation} \label{conditionnumber}
\mathrm{cond}(A_m) \leq \mathrm{cond}(-\pi \mathcal{I}+\tilde{\mathcal W}_m+{\mathcal S}_m) \leq \C,
\end{equation}
where $\C \neq \C(m)$.
\end{theorem}

According to the decomposition of the boundary $\Sigma$  and to the parametric representation (\ref{trasf}) introduced of each arc $\Sigma_i$, the double layer potential $u$ defined by (\ref{double}), solution of the Dirichlet problem (\ref{Dirichlet}), can be rewritten as
\begin{equation} \label{double2}
u(x,y)=\sum_{i=1}^{3} \int_0^1 H_i(x,y,t) \bar{\psi_i}(t) dt, \quad \forall \ (x,y) \in D
\end{equation}
where   $\bar{\psi}_i=\psi_i\circ\sigma_i$, with $\psi_i$  the double layer density function on $\Sigma_i$, and
\[
H_i(x,y,t)=\frac{\eta'_i(t)[x-\xi_i(t)]-\xi'_i(t)[y-\eta_i(t)]}{[x-\xi_i(t)]^2+[y-\eta_i(t)]^2}.
\]
Now we propose to approximate the double layer potential $u(x,y)$ in (\ref{double2}) by means of the following function
\begin{equation} \label{doubleappr}
u_m(x,y)=\sum_{i=1}^{3} \sum_{h=0}^{m+1} \lambda_h H_i(x,y,x_h) \bar{\psi}_{m,i}(x_h),
\end{equation}
obtained by replacing each function $\bar{\psi_i}$  on the right-hand side in (\ref{double2}) with the corresponding Nystr\"om interpolant $\bar{\psi}_{m,i}$ ($i$-th component of the solution $\bar{\psi}_m$ of (\ref{sistapprox})) and, then,  by approximating all the integrals using the Lobatto quadrature rule (\ref{lobatto}) on $m+2$ points.
Let us observe that the values $\bar{\psi}_{m,i}(x_h)$ involved in the formula (\ref{doubleappr}) are just the solutions of the linear system (\ref{linearsystem}).
\begin{theorem} \label{harmonicerror}
For any $(x,y) \in D$, the double layer potential $u$ defined by $(\ref{double})$, solution of the Dirichlet problem $(\ref{Dirichlet})$, and the function $u_m$ given by
(\ref{doubleappr}) satisfy the following
pointwise error estimate
\begin{equation} \label{potentialerror}
|u(x,y)-u_m(x,y)| \leq  \frac{\C}{m}\left(\frac{1}{d^2}+\frac{1}{d}\right)+\frac{\C'}{d}\left\|\bar{\psi}-\bar{\psi}_m\right\|_{\infty},
\end{equation}
where $d=\displaystyle \min_{i=1,2,3}d_i$, with $d_i=\displaystyle \min_{0 \leq t \leq 1}|(x,y)-(\xi_i(t),\eta_i(t))|$, and $\C$, $\C'$ are positive constants independent of $(x,y)$ and $m$.
\end{theorem}

Let us observe that the first addendum on the right hand side of (\ref{potentialerror}) could converge to zero
with rate greater than $1/m$ if the boundary $\Sigma \setminus P_0$ is  $(q+2)$-times differentiable, for some $q>0$.
Moreover, from the previous estimate, we can deduce that the error becomes smaller and smaller as well as the point $(x,y) \in D$ moves away from the boundary $\Sigma$.

\section{Proofs} \label{Section4}
In order to prove Theorem \ref{errorequad} we need the following result (see \cite{Szego}).
\begin{lemma}\label{coeff}
Let  $x_k$  and $\lambda_k$, $k=0,1,\ldots,m+1$, be the nodes and the coefficients of the quadrature rule defined in $(\ref{lobatto})$, respectively.
Then, setting  $\Delta x_{k}=x_{k+1}-x_k$, $k=0,1,\ldots,m$, one has
\begin{equation}  \label{nodes}
\Delta x_{k} \sim \left\{
  \begin{array}{ll}
  \displaystyle \frac{\sqrt{x_{k+1}(1-x_{k+1})}}{m}  , & k=0  \vspace{0.2cm} \\

 \displaystyle \frac{\sqrt{x_k(1-x_k)}}{m}, & k=1,\ldots,m.
  \end{array}
\right.
\end{equation}
and
\begin{equation} \label{lambda}
\lambda_k \sim \left\{
  \begin{array}{ll}
  \Delta x_{k} , &  k=0,1,\ldots,m  \vspace{0.2cm} \\
  \Delta x_{k-1}, & k=m+1.
  \end{array}
\right.
\end{equation}
\end{lemma}

\begin{proof}\emph{of Theorem \ref{errorequad}}
We can proceed analogously to the proof of
Theorem 5.1.8  in \cite{MMlibro}. Then,
it will be sufficient to prove the following  inequality
\begin{equation}\label{equ1}
\sum_{k=0}^{m+1} \lambda_k |f(x_k)| \leq \C \left( \|f\|_1+ \frac{1}{m} \|f'\varphi\|_1 \right).
\end{equation}
Indeed, since the Lobatto quadrature rule is exact for polynomials of degree at most $2m+1$, for any $P \in \PP_{2m+1}$ we can write
\begin{equation}
|e_m(f)| \leq \int_0^1 |f(x)-P(x)| dx + \sum_{k=0}^{m+1} \lambda_k |f(x_k)-P(x_k)|.
\end{equation}
Hence, by applying (\ref{equ1}) and the following inequality (\cite{KY})
\begin{equation}
\|(f-P)' \varphi\|_1\leq \C (2m+1) \|f-P\|_1+ E_{2m}(f')_{\varphi,1},
\end{equation}
 we have
\begin{eqnarray*}
|e_m(f)| & \leq \C \left[ \|f-P\|_1 + \frac{1}{m} \|(f-P)' \varphi\|_1 \right] \\
& \leq \C \left[ \|f-P\|_1 + \frac{1}{m} E_{2m}(f')_{\varphi,1} \right]
\end{eqnarray*}
from which, by taking the infimum on $P \in \PP_{2m+1}$ and using (\ref{Favard}),  we obtain
\begin{eqnarray*}
|e_m(f)| & \leq & \C \left[ E_{2m+1}(f)_1 + \frac{1}{m} E_{2m}(f')_{\varphi,1} \right] \\ & \leq & \frac{\C}{m} E_{2m}(f')_{\varphi,1}
\end{eqnarray*}
i.e. the thesis.

In order to prove (\ref{equ1}) we note that, in virtue of Lemma \ref{coeff}, we can write
\begin{equation} \label{sum}
\sum_{k=0}^{m+1} \lambda_k |f(x_k)| \leq \C  \sum_{k=0}^{m+1} \Delta x_k |f(x_k)|,
\end{equation}
where we set $\Delta x_{m+1}=\Delta x_{m}$.
Then we apply the first one of the  following ine\-qualities
\begin{equation}\label{dis1}
\left.
  \begin{array}{ll}
    (b-a) |f(a)| \\
    (b-a) |f(b)|
  \end{array}
\right\} \leq   \left[ \int_a^b |f(t)| dt + (b-a) \int_a^b |f'(t)| dt \right]
\end{equation}
with  $a=x_k$ and $b={x_{k+1}}$ in order to estimate the terms on the right hand side of (\ref{sum}) with $k=0,1,\ldots,m$ and we have
\begin{eqnarray} \label{addend1}
\Delta x_k |f(x_k)| & \leq & \left[ \int_{x_k}^{x_{k+1}} |f(x)| dx + \Delta x_k \int_{x_k}^{x_{k+1}} |f'(x)| dx \right] \nonumber \\
 & \leq & \C  \left[ \int_{x_k}^{x_{k+1}} |f(x)| dx + \frac{1}{m} \int_{x_k}^{x_{k+1}} |(f' \varphi)(x)| dx \right]
\end{eqnarray}
being, for $x \in [x_k,x_{k+1}]$, $x_k \sim  x \sim x_{k+1}$ and $1-x_k \sim 1-x \sim 1-x_{k+1}$. \newline
For the term $k=m+1$, we can apply  the second inequality of (\ref{dis1})  and obtain
\begin{equation} \label{addend2}
\Delta x_{m+1} |f(x_{m+1})| \leq \C \left[ \int_{x_m}^{x_{m+1}} |f(x)| dx + \frac{1}{m} \int_{x_m}^{x_{m+1}} |(f' \varphi)(x)| dx \right].
\end{equation}
Finally, summing up on $k=0,1,...,m+1$ inequalities (\ref{addend1})-(\ref{addend2}), we can deduce (\ref{equ1}). \qquad
\end{proof}

\begin{proof}\emph{of Theorem \ref{theoremsolv}}
We first prove that  ${\mathcal W}:{\mathcal X} \rightarrow {\mathcal X}$ is a bounded operator and satisfies
\begin{equation} \label{normW}
\|{\mathcal W}\| < \pi.
\end{equation}
From well known results (see, for instance, \cite[p. 393]{A}) it follows that for any array of functions
$\rho=(\rho_1,\rho_2,\rho_3)^T \in  {\mathcal X}$ one has that ${\mathcal W}\rho \in C([0,1])^{^3}$.
Moreover, it is easy to see that ${\mathcal W}\rho \in {\mathcal X}$ and
if $\|\rho\|_{\infty} \leq 1$,
\[
\|{\mathcal W}\rho\|_{\infty} \leq |\chi|\pi<\pi.
\]
Therefore, since, for $\mathcal{I}:{\mathcal X} \to {\mathcal X}$, $\|-\pi \mathcal{I}\|=\pi$, by applying the geometric series theorem we deduce that  $(-\pi \mathcal{I}+ {\mathcal W})^{-1}$   exists and is a bounded operator
on ${\mathcal X}$ into ${\mathcal X}$
with
\[
\|(-\pi \mathcal{I}+ {\mathcal W})^{-1}\| \leq \frac{1}{\pi - \|{\mathcal W}\|}.
\]
Consequently,  we can reformulate equation (\ref{sistop1}) as
\begin{equation} \label{sistop2}
\bar{\psi}+(-\pi \mathcal{I}+ {\mathcal W})^{-1} {\mathcal{S}} \bar{\psi}=(-\pi \mathcal{I}+ {\mathcal{W}})^{-1} \bar{g}.
\end{equation}
Now, let us note that the operator ${\mathcal{S}}$ also maps ${\mathcal X}$ into ${\mathcal X}$ and it is compact since it is a matrix of compact operators.
Hence $(-\pi \mathcal{I}+ {\mathcal W})^{-1} {\mathcal{S}}$ is a compact operator, too. Thus for equation (\ref{sistop2})
the Fredholm alternative holds true and from the hypothesis it follows that the system (\ref{sistop1}) is unisolvent in ${\mathcal X}$ for each right-hand side $\bar{g} \in {\mathcal X}$. \newline
In particular, if $\bar{g} \in \tilde{\mathcal X}$ then the vector
$\bar{\psi}=(-\pi \mathcal{I}+ {\mathcal W}+{\mathcal S})^{-1}\bar{g}$ also belongs to the subspace $\tilde{\mathcal X}$. In fact, since the operator $-\pi \mathcal{I}+ {\mathcal W}+{\mathcal S}=-\pi+\mathcal{K}$ is invertible in $\tilde{\mathcal X}$ (see (\ref{opeta})), there exists an array $\bar{\varphi} \in \tilde{\mathcal X}\subset {\mathcal X}$ such that $\bar{\varphi}=(-\pi \mathcal{I}+ {\mathcal W}+{\mathcal S})^{-1}\bar{g}$. Then, by the assumption $\bar{\psi}=\bar{\varphi}$ follows.
\end{proof}

In order to be able to prove Theorem \ref{lemmaWm} we need to prove the following two lemmas.
\begin{lemma}\label{propLm}
Let
\[
L(t,s)=-\frac{s \sin {(\chi \pi)}}{s^2+2ts \cos {(\chi \pi)}+t^2}, \quad t,s \in [0,1],
\]for some $\chi \in \RR$, $|\chi|<-1$, and let  $e_m$   be the functional defined as in $(\ref{lobatto})$.
Then, for each $s \in  (0,1]$ one has
$$e_m(L(\cdot,s)) \leq \C \frac{r!}{m^r} \frac{1}{s^{r/2}},$$
where $r \in \NN$ and $\C \neq \C(m)$. 
\end{lemma}
\begin{proof}
At first let us note that
$$L(t,s)=
-\frac{1}{2i} \left[ \frac{1}{t+e^{-i\chi \pi}s}-\frac{1}{t+e^{i \chi \pi}s}\right],$$
from which, for any fixed integer $r$, we have
\begin{eqnarray*}
\frac{\partial^r}{\partial t^r} L(t,s)&= &-\frac{1}{2i}\left[ \frac{(-1)^r r!}{(t+e^{-i\chi \pi}s)^{r+1}}-\frac{(-1)^r r!}{(t+e^{i\chi \pi}s)^{r+1}} \right] \\ & =&-
\frac{(-1)^r r!}{2i} \frac{(t+e^{i\chi \pi}s)^{r+1}-(t+e^{-i\chi \pi}s)^{r+1}}{\left[(t+e^{-i\chi \pi}s) (t+e^{i\chi \pi}s)\right]^{r+1}}  \\ & = &-\frac{(-1)^r r!}{2i}  \left[ \frac{ \displaystyle \sum_{k=0}^{r+1} \left(\begin{array}{c} r+1 \\ k \end{array} \right) t^k s^{r+1-k}[e^{i \chi \pi (r+1-k)}-e^{-i\chi \pi (r+1-k)}]}{(t^2+2ts \cos{\chi \pi}+s^2)^{r+1}}\right] \\ &= &-(-1)^r r!  \left[ \frac{ \displaystyle \sum_{k=0}^{r+1} \left(\begin{array}{c} r+1 \\ k \end{array} \right) t^k s^{r+1-k}\sin{(\chi \pi (r+1-k))}}{(t^2+2ts \cos{\chi \pi}+s^2)^{r+1}}\right].
\end{eqnarray*}
Then
\begin{eqnarray} \label{ineqderivL}
\left| \frac{\partial^r}{\partial t^r} L(t,s)\right| & \leq & r! \frac{\left| \displaystyle \sum_{k=0}^{r+1} \left(\begin{array}{c} r+1 \\ k \end{array} \right) t^k s^{r+1-k} \right|}{(t^2+2ts \cos{\chi \pi}+s^2)^{r+1}}  \\ &= &r! \frac{(t+s)^{r+1}}{(t^2+2ts \cos{\chi \pi}+s^2)^{r+1}} \nonumber
\end{eqnarray}
and, consequently,
\begin{eqnarray*}
\left \|\frac{\partial^r}{\partial t^r} L(\cdot,s) \varphi^r \right \|_1 &=& \int_0^1 \left|\frac{\partial^r}{\partial t^r} L(t,s)\right| \varphi^r(t) dt \\ & \leq & r! \int_0^1 \frac{(t+s)^{r+1} t^{r/2}}{(t^2+2ts \cos{\chi \pi}+s^2)^{r+1}} dt.
\end{eqnarray*}
Now, setting $t=sx$, we can deduce
\begin{eqnarray} \label{estimate}
\left \|\frac{\partial^r}{\partial t^r} L(\cdot,s) \varphi^r \right \|_1 & \leq & \frac{r!}{s^{r/2}} \int_0^{1/s} \frac{(x+1)^{r+1} x^{r/2}}{(x^2+2x \cos{\chi \pi}+1)^{r+1}} dx \nonumber \\ & \leq & \frac{r!}{s^{r/2}} \int_0^\infty \frac{(x+1)^{r+1} x^{r/2}}{(x^2+2x \cos{\chi \pi}+1)^{r+1}} dx   \leq \C \frac{r!}{s^{r/2}}.
\end{eqnarray}
Thus, by applying Corollary \ref{corquad}, for $r \in \NN$, it results
\begin{equation}\label{sticol}
|e_m(L(\cdot,s))|  \leq \frac{\C}{m^r} E_{2m+1-r} \left( \frac{\partial^r}{\partial t^r} L(\cdot,s)\right)_{\varphi^r,1}  \leq \frac{\C}{m^r} \left \| \frac{\partial^r}{\partial t^r} L(\cdot,s) \varphi^r \right\|_1
\end{equation}
and, combining (\ref{sticol}) with (\ref{estimate}), the thesis follows.
\end{proof}

\begin{lemma}\label{densesubspace}
Let ${\mathcal X}$ be the space of functions defined in $(\ref{space})$ and
\begin{equation} \label{ptilde}
\tilde{\PP}^{^3}=\PP^{^3}\cap {\mathcal X}
\end{equation}
where $\PP$ is the set of all polynomials on $[0,1]$. Then $\tilde{\PP}^{^3}$
is a dense subspace in ${\mathcal X}$.
\end{lemma}
\begin{proof}
First, let us prove that $\overline{\tilde{\PP}^{^3}} \subseteq {\mathcal X}$ ($\overline{\tilde{\PP}^{^3}}$ denoting the closure of $\tilde{\PP}^{^3}$). \newline
Let $\varphi=(\varphi_1,\varphi_2,\varphi_3) \in \overline{\tilde{\PP}^{^3}}$ and $\{p_m\}_m$, with $p_m=(p_{m,1},p_{m,2},p_{m,3})
\in \tilde{\PP}^{^3}$, be a sequence convergent to $\varphi$. Then
\[
\lim_m p_{m,i}=\varphi_i, \quad \forall i \in\{1,2,3\}
\]
and, consequently, one has $\varphi \in {\mathcal X}$ since
\[
\varphi_1(0)=\lim_m p_{m,1}(0)=\lim_m p_{m,2}(0)=\varphi_{2}(0).
\]

Viceversa, we are going to prove that ${\mathcal X} \subseteq \overline{\tilde{\PP}^{^3}}$. Since $\PP^{^3}$ is a dense subspace of the space $C([0,1])^{^3}$, for a given  $\varphi=(\varphi_1,\varphi_2,\varphi_3) \in {\mathcal X}$ there exists a sequence $\{p_m\}_m$, with $p_m=(p_{m,1},p_{m,2},p_{m,3}) \in {\PP}^{^3}$
such that
\begin{equation} \label{limitpm1}
\lim_m p_m=\varphi,
\end{equation}
 from which it follows
\begin{equation} \label{limitpm2}
\lim_m p_{m,1}(0)=\varphi_1(0)=\varphi_{2}(0)=\lim_m p_{m,2}(0).
\end{equation}
Starting from the sequence $\{p_m\}_m$,  we introduce a new sequence $\{q_m\}_m$, with $q_m=(q_{m,1},q_{m,2},q_{m,3}) \in {\PP}^{^3}$ defined as follows
\[
q_{m,i}(x)=\left\{\begin{array}{lr}
xp_{m-1,i}(x)+(1-x)p_{m-1,i}(x)\displaystyle\frac{p_{m-1,i+1}(0)}{\displaystyle \lim_mp_{m,i}(0)}, & \quad i=1 \\
xp_{m-1,i}(x)+(1-x)p_{m-1,i}(x)\displaystyle \frac{p_{m-1,i-1}(0)}{\displaystyle  \lim_mp_{m,i}(0)}, & \quad i= 2\\
p_{m,i}(x), & \quad i =3
\end{array} \right. .
\]
Then,  taking into account (\ref{limitpm2}), one has
\[
q_{m,1}(0)=p_{m-1,1}(0)\displaystyle\frac{p_{m-1,2}(0)}{\displaystyle  \lim_mp_{m,1}(0)}=p_{m-1,2}(0)\displaystyle \frac{p_{m-1,1}(0)}{\displaystyle \lim_mp_{m,2}(0)}=q_{m,2}(0),
\]
i.e. $q_m \in \tilde{\PP}^{^3}$. Moreover, by using (\ref{limitpm1}) and (\ref{limitpm2}), again,
it can be easily proved that
\[
\lim_mq_{m,i}=\varphi_i, \quad \forall i \in \{1,2,3\},
\]
i.e. $\lim_m q_m=\varphi,$ from which the thesis follows.
\end{proof}

\noindent \emph{Proof of Theorem \ref{lemmaWm}}
We start by showing that the operators $\tilde{\mathcal W}_m$ map ${\mathcal X}$ into ${\mathcal X}$. To this aim it is sufficient to
observe that for any array of functions ${\rho}\in{\mathcal X}$ one has that $\tilde{\mathcal W}_m \rho\in C([0,1])^{^3}$ and $(\tilde{\mathcal W}_m{\rho})(0)=({\mathcal W}{\rho})(0)$. \newline
Now we are going to prove (\ref{limnormWm}). Let ${\rho}=(\rho_1,\rho_2,\rho_3)^T \in {\mathcal X}$ such that  $\|{\rho}\|_{\infty} \leq 1$.
One has
\begin{equation} \label{equality1}
\|\tilde{\mathcal W}_m {\rho}\|_{\infty}
 =  \max\left\{\sup_{s \in \left[0,\frac{c}{m^{2-2\epsilon}}\right]}\|(\tilde{\mathcal W}_m {\rho})(s)\|_{\infty}, \sup_{s \in\left[\frac{c}{m^{2-2\epsilon}},1\right]}\|(\tilde{\mathcal W}_m {\rho})(s)\|_{\infty}\right\}.
\end{equation}
Now, by  (\ref{opmodified}), we deduce
\begin{eqnarray*}
\sup_{s \in\left[\frac{c}{m^{2-2\epsilon}},1\right]}\|(\tilde{\mathcal W}_m {\rho})(s)\|_{\infty}&=&\sup_{s\in\left[\frac{c}{m^{2-2\epsilon}},1\right]}
\max\left\{\left|(-\pi+\bar{\Omega}_1(s))\rho_1(s)+(\mathcal L_m\rho_2)(s)\right|, \right.\\
& & \left.\left|(-\pi+\bar{\Omega}_2(s))\rho_2(s)+(\mathcal L_m\rho_1)(s)\right|\right\}.
\end{eqnarray*}
Being
\[
\sum_{h=0}^{m+1} \lambda_h L(x_h,s)=\int_0^1L^{i,j}(t,s)dt-e_m(L(\cdot,s)), \quad s \in [0,1],
\]
where $e_m$ denotes the error of the Lobatto quadrature formula in (\ref{lobatto}), for $s \in \left[\displaystyle \frac{c}{m^{2-2\epsilon}},1\right]$,
we can write that
\begin{eqnarray*}
\left|(-\pi+\bar{\Omega}_1(s))\rho_1(s)+(\mathcal L_m\rho_2)(s)\right|  &  &\leq  \|{\rho}\|_{\infty}\left(
\left|-\pi+\bar{\Omega}_1(s)\right|+\left|\sum_{h=0}^{m+1} \lambda_h L(x_h,s)\right|\right)\\
& & \leq    \left|-\pi+\bar{\Omega}_1(s)\right|+\left|\int_0^1L(t,s)dt\right|+\left|e_m(L(\cdot,s))\right|\\
& & \leq  \sup_{s \in[0,1]}\left(\left|-\pi+\bar{\Omega}_1(s)\right|+\int_0^1\left|L(t,s)\right|dt\right) \\
& & \hspace{1 cm}  +\sup_{s \in\left[\frac{c}{m^{2-2\epsilon}},1\right]}\left|e_m(L(\cdot,s))\right|\\
& & = |\chi|\pi+\sup_{s \in\left[\frac{c}{m^{2-2\epsilon}},1\right]}\left|e_m(L(\cdot,s))\right|
\end{eqnarray*}
and, similarly,
\[
\left|(-\pi+\bar{\Omega}_{2}(s))\rho_2(s)+(\mathcal L_m\rho_1)(s)\right|\leq |\chi|\pi+\sup_{s \in\left[\frac{c}{m^{2-2\epsilon}},1\right]}\left|e_m(L(\cdot,s))\right|.
\]
Then, in virtue of Lemma \ref{propLm}, for any $r \in \NN$, we have
\begin{equation} \label{estimate1}
\sup_{s \in\left[\frac{c}{m^{2-2\epsilon}},1\right]}\|(\tilde{\mathcal W}_m {\rho})(s)\|_{\infty}< \pi+\C \frac{r!}{m^{r\epsilon}}, \quad \C=\C(r).
\end{equation}
It remains to estimate
$\displaystyle \sup_{s \in\left[0,\frac{c}{m^{2-2\epsilon}}\right]}\|(\tilde{\mathcal W}_m {\rho})(s)\|_{\infty}$. Taking into account the definition (\ref{opmodified}) of the operator $\tilde{\mathcal W}_m$, we can write
\begin{eqnarray*}
& &\hspace*{-0.8cm}\sup_{s \in \left[0,\frac{c}{m^{2-2\epsilon}}\right]}\|(\tilde{\mathcal W}_m {\rho})(s)\|_{\infty} \\
&\leq & \frac{m^{2-2\epsilon}}{c} \hspace*{-0.1cm}\sup_{s \in\left[0,\frac{c}{m^{2-2\epsilon}}\right]} \left\{s\left\|({\mathcal W}_m {\rho})\left(\frac{c}{m^{2-2\epsilon}}\right)\right\|_{\infty}+\left(\frac{c}{m^{2-2\epsilon}}-s\right)\left\|({\mathcal W}
{\rho})(0)\right\|_{\infty}\right\}\\
&=& \max\left\{\left\|\left({\mathcal W}_m {\rho}\right)\left(\frac{c}{m^{2-2\epsilon}}\right)\right\|_{\infty},\left\|\left({\mathcal W}{\rho}\right)(0)\right\|_{\infty}\right\}\\
&\leq &\max\left\{\sup_{s \in\left[\frac{c}{m^{2-2\epsilon}},1\right]}\left\|\left({\mathcal W}_m {\rho}\right)(s)\right\|_{\infty},\left\|\left({\mathcal W}{\rho}\right)(0)\right\|_{\infty}\right\}.
\end{eqnarray*}
Now since $\left\|\left({\mathcal W}{\rho}\right)(0)\right\|_{\infty} \leq |\chi|\pi$
and (\ref{estimate1}) holds true, we can conclude that, for any $r \in \NN$,
\begin{equation} \label{estimate2}
\sup_{s \in\left[0,\frac{c}{m^{2-2\epsilon}}\right]}\|(\tilde{\mathcal W}_m {\rho})(s)\|_{\infty} < \pi+\C \frac{r!}{m^{r\epsilon}}.
\end{equation}
Finally, combining (\ref{equality1}), (\ref{estimate1}) and (\ref{estimate2}) we have that
\begin{equation}  \label{estnormWm}
\|\tilde{\mathcal W}_m \| < \pi + \C \frac{r!}{m^{r\epsilon}}, \quad r \in \NN
\end{equation}
 i.e. (\ref{limnormWm}). In order to prove (\ref{convpointnormWm}) we want to apply the Banach-Steinhaus theorem (see, for instance, \cite[p. 517]{A}). First, we recall that the subset  $\tilde{\PP}^{^3}=\PP^{^3}\cap {\mathcal X}$ is a dense subspace of ${\mathcal X}$ (see Lemma \ref{densesubspace}). Then we are going to show that
\begin{equation} \label{assertion(i)}
\displaystyle \lim_{m \to \infty} \| (\tilde{\mathcal W}_m- {\mathcal W}) {\rho}\|_\infty = 0, \quad  \forall \  \rho=(\rho_1,\rho_2, \rho_3)^T \in \tilde{\PP}^{^3},
\end{equation}
and that the operators $\tilde{\mathcal W}_m:{\mathcal X} \to {\mathcal X}$ are uniformly bounded with respect to $m$, i.e.
\begin{equation} \label{assertion(ii)}
\displaystyle \sup_m \|\tilde{\mathcal W}_m \| < \infty.
\end{equation}
Assertion (\ref{assertion(ii)})  follows from (\ref{estnormWm}).
In order to prove (\ref{assertion(i)}),  noting that
\[
\| (\tilde{\mathcal W}_m- {\mathcal W}) {\rho}\|_\infty =\max\left\{\sup_{s \in \left[0,\frac{c}{m^{2-2\epsilon}}\right]} \hspace{-0.5 cm}\|(\tilde{\mathcal W}_m- {\mathcal W}) {\rho}(s)\|_{\infty}, \sup_{s \in\left[\frac{c}{m^{2-2\epsilon}},1\right]}\hspace{-0.5 cm} \|(\tilde{\mathcal W}_m- {\mathcal W}){\rho}(s)\|_{\infty}\right\},
\]
we are going to show that both the terms into the braces converges to zero when $m \to \infty$.
For the second one it is sufficient to show that
\begin{equation} \label{Lmij-Lij}
\lim_{m \to \infty} \sup_{s \in\left[ \frac{c}{m^{2-2\epsilon}},1\right]}| ({\mathcal L}_m- {\mathcal L}) {p}(s)|= 0,  \quad \forall \ p \in \PP.
\end{equation}
Fixed $p \in \PP$, by applying the  error estimate (\ref{errorequad2}) for the Lobatto quadrature formula to the function $L(\cdot,s)p$, we have, for any $r\in \NN$
\begin{eqnarray*}
|({\mathcal L}_m-{\mathcal L}) {p}(s)| \leq \frac {\C}{m^r}\int_0^1\left|\frac{\partial^r}{\partial t^r}\left(L(t,s)p(t)\right)\right|\varphi^r(t)dt.
\end{eqnarray*}
But
\begin{eqnarray*}
 \int_0^1\left|\frac{\partial^r}{\partial t^r}\left(L(t,s)p(t)\right)\right|\varphi^r(t)dt
& = &\int_0^1\left|\sum_{k=0}^r\left(\begin{array}{c}r\\k\end{array}\right)\frac{\partial^k}{\partial t^k}L(t,s)p^{(r-k)}(t)\right|\varphi^r(t)dt\\
& \leq & \sum_{k=0}^r\left(\begin{array}{c}r\\k\end{array}\right)\int_0^1\left|\frac{\partial^k}{\partial t^k}L(t,s)\right|\varphi^k(t)\left|p^{(r-k)}(t)\right|\varphi^{r-k}(t)dt\\
& \leq & \C \sum_{k=0}^r\left(\begin{array}{c}r\\k\end{array}\right)\int_0^1\left|\frac{\partial^k}{\partial t^k}L(t,s)\right|\varphi^k(t)dt
\end{eqnarray*}
and, taking into account the estimate (\ref{estimate}), we get for $\displaystyle \frac{c}{m^{2-2\epsilon}}\leq s\leq 1$
$$\int_0^1\left|\frac{\partial^r}{\partial t^r}\left(L(t,s)p(t)\right)\right|\varphi^r(t)dt \leq  \C \sum_{k=0}^r\left(\begin{array}{c}r\\k\end{array}\right)\frac{k!}{s^{k/2}}
\leq \C 2^r r! \ m^{r(1-\epsilon)},$$
with $\C=\C(r,p)$. Hence, we deduce that
\[
|({\mathcal L}_m-{\mathcal L}) {p}(s)| \leq \frac {\C 2^r \ r! } {m^{r\epsilon}}
\]
from which (\ref{Lmij-Lij}) follows. \newline
Now let us consider $s \in \left[0,\displaystyle \frac{c}{m^{2-2\epsilon}}\right]$. By the definition (\ref{opmodified}), we have
\begin{eqnarray*}
& &\|(\tilde{\mathcal W}_m-{\mathcal W}) {\rho}(s)\|_{\infty} \\
&  & = \left\|\frac{m^{2-2\epsilon}}{c}\left[s({\mathcal W}_m{\rho})\left(\frac{c}{m^{2-2\epsilon}}\right)+\left(\frac{c}{m^{2-2\epsilon}} -s\right)({\mathcal W}{\rho})(0) \right] - ({\mathcal W} {\rho})(s)\right\|_{\infty}\\
&  & \leq  \frac{m^{2-2\epsilon}}{c} s   \left\| ({\mathcal W}_m {\rho})\left(\frac{c}{m^{2-2\epsilon}}\right)- ({\mathcal W} {\rho})\left(\frac{c}{m^{2-2\epsilon}}\right)\right\|_{\infty} \\
&  &  \hspace{0.5cm}+\frac{m^{2-2\epsilon}}{c} s  \left\| ({\mathcal W} {\rho})\left(\frac{c}{m^{2-2\epsilon}}\right)- ({\mathcal W} {\rho})(0)\right\|_{\infty} +\left\|({\mathcal W}{\rho})(0)-({\mathcal W} {\rho})(s)\right\|_{\infty}.
\end{eqnarray*}
For the first addendum we can write
\begin{eqnarray*}
& &\sup_{s \in \left[0,\displaystyle \frac{c}{m^{2-2\epsilon}}\right]}\frac{m^{2-2\epsilon}}{c} s   \left\| ({\mathcal W}_m {\rho})\left(\frac{c}{m^{2-2\epsilon}}\right)- ({\mathcal W} {\rho})\left(\frac{c}{m^{2-2\epsilon}}\right)\right\|_{\infty}\\
& &\leq \sup_{s \in\left[\frac{c}{m^{2-2\epsilon}},1\right]} \|({\mathcal W}_m-{\mathcal W}) {\rho}(s)\|_{\infty}
\end{eqnarray*}
and, then, from (\ref{Lmij-Lij}) it follows that
\begin{equation} \label{limsup2}
\lim_{m \to \infty} \sup_{s \in \left[0,\displaystyle \frac{c}{m^{2-2\epsilon}}\right]}\frac{m^{2-2\epsilon}}{c} s   \left\| ({\mathcal W}_m {\rho})\left(\frac{c}{m^{2-2\epsilon}}\right)- ({\mathcal W} {\rho})\left(\frac{c}{m^{2-2\epsilon}}\right)\right\|_{\infty}=0.
\end{equation}
For the second addendum we get
\begin{eqnarray*}
& &\lim_{m \to \infty}\hspace*{-0.2cm} \sup_{s \in \left[0,\displaystyle \frac{c}{m^{2-2\epsilon}}\right]}\frac{m^{2-2\epsilon}}{c} s  \left\| ({\mathcal W} {\rho})\left(\frac{c}{m^{2-2\epsilon}}\right)- ({\mathcal W} {\rho})(0)\right\|_{\infty}\\
& &= \lim_{m \to \infty} \left\| ({\mathcal W} {\rho})\left(\frac{c}{m^{2-2\epsilon}}\right)- ({\mathcal W} {\rho})(0)\right\|_{\infty}=0
\end{eqnarray*}
since ${\mathcal W} {\rho} \in C([0,1])^3$.
Finally, being also
\[
\lim_{m \to \infty} \sup_{s \in \left[0,\displaystyle \frac{c}{m^{2-2\epsilon}}\right]}\left\|({\mathcal W} {\rho})(0)-({\mathcal W} {\rho})(s)\right\|=0
\]
we can conclude that
\begin{equation}\label{limsup3}
\lim_{m \to \infty} \sup_{s \in \left[0,\displaystyle \frac{c}{m^{2-2\epsilon}}\right]}\|(\tilde{{\mathcal W}}_m- {\mathcal W}) {\rho}(s)\|_{\infty} =0.
\end{equation}
This completes the proof.


\begin{proof}\emph{of Theorem \ref{opercompatti}}
At first, let us note that the operators ${\mathcal S}_m$ map ${\mathcal X}$ into ${\mathcal X}$ and the set $\{{\mathcal S}_m\}_m$
is collectively compact if  the sets  of operators $\{\mathcal{K}_m^{i,j}\}_m$,  for any fixed couple of indices $(i,j)$ such that $i=j=1$, $i=j=2$, $i=3$ or $j=3$,
and $\{\mathcal{M}_m^{i,j}\}_m$, for $(i,j)=(1,2)$ and $(i,j)=(2,1)$,  are collectively compact.
 Moreover, by  definition (\ref{MatrixSm}), it results that $\forall  \rho=(\rho_1,\rho_2,\rho_3) \in {\mathcal X}$, if
\begin{equation}\label{convpunt}
\lim_{m \to \infty} \|(\mathcal{K}_m^{i,j}-\mathcal{K}^{i,j}) \rho_j\|_\infty=0,
\end{equation}
and
\begin{equation}\label{convpunt2}
\lim_{m \to \infty} \|(\mathcal{M}_m^{i,j}-\mathcal{M}^{i,j}) \rho_j\|_\infty=0,
\end{equation}
then $\displaystyle \lim_{m \to \infty} \|({\mathcal S}_m -{\mathcal S}){\rho}\|_{\infty}=0$.

Now, the limit conditions (\ref{convpunt}), (\ref{convpunt2})
can be immediately deduced taking into account (\ref{Kmij}), (\ref{Mmij}),  
the continuity of the kernels $K^{i,j}(\cdot,s)$ and $M^{i,j}(\cdot,s)$, and the convergence of the Lobatto quadrature rule on the set $C([0,1])$.
From this, by applying standard arguments, (see, for instance, \cite[Theorem 12.8]{K}) it also follows that the sets $\{\mathcal{K}_m^{i,j}\}_m$
and $\{\mathcal{M}_m^{i,j}\}_m$, with $i$ and $j$ as specified above, are collectively compact and the proof is complete.
\end{proof}

\begin{proof}\emph{of Theorem \ref{maintheorem}}
First of all we observe that by (\ref{assertion(ii)}) and (\ref{convpointnormWm}) we can deduce that the operators
 $-\pi \mathcal{I}+\tilde{\mathcal W}_m:{\mathcal X} \to :{\mathcal X}$ are bounded and pointwise convergent to $-\pi \mathcal{I}+{\mathcal W}$. \newline
Moreover, from (\ref{limnormWm}),
in virtue of the geometric series theorem, it follows  that for sufficiently large $m$ the operators $(-\pi \mathcal{I}+{\mathcal W}_m)^{-1}:{\mathcal X} \to {\mathcal X}$ exist and are uniformly bounded with
$$\|(-\pi \mathcal{I}+\tilde{\mathcal W}_m)^{-1}\| \leq \frac{1}{\pi-\displaystyle \sup_m\|\tilde{\mathcal W}_m\|}$$
(see also (\ref{assertion(ii)})).
Now taking into account Theorem \ref{opercompatti}, it results (see, for instance, Theorem 10.8 and Problem 10.3 in \cite{K}) that for
sufficiently large $m$ the operators  $$(-\pi \mathcal{I}+\tilde{\mathcal W}_m +{\mathcal S}_m)^{-1}:{\mathcal X} \to {\mathcal X}$$ exist
and are uniformly bounded, i.e. the method is stable.

From this, since
\[
\bar{{\psi}}-\bar{{\psi}}_m=(-\pi \mathcal{I}+{\tilde{\mathcal W}}_m +{{\mathcal S}}_m)^{-1}
\left[(\tilde{\mathcal W}_m-{\mathcal W})\bar{{\psi}}+({\mathcal S}_m-{\mathcal S})\bar{{\psi}}\right]
\]
we immediately deduce (\ref{errorestimate}).

Finally, in order to estimate the first term in the brackets on the right hand side of (\ref{errorestimate}),
we will consider separately the cases $s \in \left[0,\displaystyle \frac {c}{m^{2-2\epsilon}}\right]$ and
$s \in \left[\displaystyle \frac {c}{m^{2-2\epsilon}},1\right]$. \newline
For $s \in \left[0,\displaystyle \frac {c}{m^{2-2\epsilon}}\right]$,  by proceeding as in the proof of Theorem \ref{lemmaWm}, we obtain
\begin{eqnarray} \label{inequalities}
\|(\tilde{\mathcal W}_m-{\mathcal W})\bar{\psi}(s)\|_{\infty} & \leq &
\left\| ({\mathcal W}_m \bar{\psi})\left(\frac{c}{m^{2-2\epsilon}}\right)- ({\mathcal W}\bar{\psi})\left(\frac{c}{m^{2-2\epsilon}}\right)\right\|_{\infty}  \nonumber\\
& \hspace{0.5 cm} & + \left\| ({\mathcal W} \bar{\psi})\left(\frac{c}{m^{2-2\epsilon}}\right)- ({\mathcal W} \bar{\psi})(0)\right\|_{\infty} \nonumber\\
& & \hspace{0.5 cm} +\left\|({\mathcal W} \bar{\psi})(0)-({\mathcal W} \bar{\psi})(s)\right\|_{\infty}.
\end{eqnarray}
We first consider the second addendum on the right-hand side in (\ref{inequalities})
\begin{eqnarray*}
\left\| ({\mathcal W} \bar{\psi})\left(\frac{c}{m^{2-2\epsilon}}\right)- ({\mathcal W} \bar{\psi})(0)\right\|_{\infty}
 &=&\max\left\{\left|( ({\mathcal L}\bar{\psi}_2)\left(\frac{c}{m^{2-2\epsilon}}\right)+\chi\pi\bar{\psi}_1(0)\right|\right.,\\
 & &\left. \left|({\mathcal L}\bar{\psi}_1)\left(\frac{c}{m^{2-2\epsilon}}\right)+\chi\pi \bar{\psi}_2(0)\right|\right\}.
\end{eqnarray*}
Recalling that $\bar{\psi}_1(0)=\bar{\psi}_2(0)$ ($\bar{{\psi}} \in \mathcal{X}$) and using the change of variable
$t=\displaystyle \frac {c}{m^{2-2\epsilon}}\tau$ we can write
\begin{eqnarray*}
\left|(\mathcal L\bar{\psi}_2)\left(\frac{c}{m^{2-2\epsilon}}\right)+\pi \chi\bar{\psi}_1(0)\right| \\
& & \hspace*{-5cm} =\left|-\int_0^{\frac{m^{2-2\epsilon}}{c}}\frac{\sin{(\chi\pi)}}{\tau^2+2\tau\cos{(\chi\pi)}+1}\bar{\psi}_2\left(\frac{c}{m^{2-2\epsilon}}\tau\right)d\tau +\chi\pi\bar{\psi}_2(0)\right|\\
& & \hspace*{-5cm} \leq  \left|-\int_0^{\frac{m^{2-2\epsilon}}{c}}\frac{\sin{(\chi\pi)}}{\tau^2+2\tau\cos{(\chi\pi)}+1}\left[\bar{\psi}_2\left(\frac{c}{m^{2-2\epsilon}}\tau\right)-
\bar{\psi}_2(0)\right]d\tau\right|\\
& & \hspace*{-5cm} \hspace{0.5cm}+ \left|-\int_0^{\frac{m^{2-2\epsilon}}{c}}\frac{\sin{(\chi\pi)}}{\tau^2+2\tau\cos{(\chi\pi)}+1}d\tau+\pi \chi \right| \cdot |\bar{\psi}_2(0)|
=:A_1+A_2.
\end{eqnarray*}
Now, taking into account the behavior of the solution $\bar{\psi}$ (see (\ref{behavioursolution}))
around the point $s=0$  and setting $\beta=\displaystyle\frac{1}{1+|\chi|}$, we have
\begin{eqnarray*}
A_1 & \leq & \C \int_0^{\frac{m^{2-2\epsilon}}{c}}\frac{\sin{(|\chi|\pi)}}{\tau^2+2\tau\cos{(\chi\pi)}+1}\left|\frac{c}{m^{2-2\epsilon}}\tau\right|^{\beta}d\tau\\
& \leq & \C \left(\frac{c}{m^{2-2\epsilon}}\right)^{\beta}\int_0^{\infty}\frac{\sin{(|\chi|\pi)}\tau^{\beta}}{\tau^2+2\tau\cos{(\chi\pi)}+1}d\tau\\
& \leq & \C \left(\frac{c}{m^{2-2\epsilon}}\right)^{\beta}
\end{eqnarray*}
and
\begin{eqnarray*}
A_2 & = & \left|-\arctan{\left(\frac{\displaystyle \frac{c}{m^{2-2\epsilon}}+\cos{(\chi\pi)}}{\sin{(\chi\pi)}}\right)}+\frac{\pi}{2}\right|\cdot |\bar{\psi}_2(0)|\\
& = & O\left(\frac{1}{m^{2-2\epsilon}}\right), \quad \mathrm{as} \ m \to \infty.
\end{eqnarray*}
Since the exponent $\beta$ satisfies $\displaystyle \frac 1 2 < \beta < 1$, we can conclude that
\begin{equation}
\left|(\mathcal L\bar{\psi}_2)\left(\frac{c}{m^{2-2\epsilon}}\right)+\pi \chi\bar{\psi}_1(0)\right| \leq \C \left(\frac{1}{m^{2-2\epsilon}}\right)^{\beta}.
\end{equation}
Following the same arguments, it can be proved that
\begin{equation}
\left|(\mathcal L\bar{\psi}_1)\left(\frac{c}{m^{2-2\epsilon}}\right)+\pi \chi\bar{\psi}_2(0)\right| \leq \C \left(\frac{1}{m^{2-2\epsilon}}\right)^{\beta}.
\end{equation}
Therefore
\begin{equation} \label{est1}
\left\| ({\mathcal W} \bar{\psi})\left(\frac{c}{m^{2-2\epsilon}}\right)- ({\mathcal W}\bar{\psi})(0)\right\|_{\infty}
\leq \C \left(\frac{1}{m^{2-2\epsilon}}\right)^{\beta}.
\end{equation}
In order to estimate the third term on the right-hand side of (\ref{inequalities}) one can proceed analogously to the proof of estimate (\ref{est1})
 and get
\begin{equation} \label{est2}
\left\|({\mathcal W} \bar{\psi})(0)-({\mathcal W}\bar{\psi})(s)\right\|_{\infty} \leq \C \left(\frac{1}{m^{2-2\epsilon}}\right)^{\beta}.
\end{equation}
It remains to estimate the first addendum in (\ref{inequalities}). Let us consider now $s \in \left[\displaystyle \frac {c}{m^{2-2\epsilon}},1\right]$.
In this case, by the definition we can write
\begin{equation}
\|(\tilde{\mathcal W}_m- {\mathcal W}) \bar{\psi}(s)\|_{\infty}
= \max\left\{\left|(\mathcal L_m-\mathcal L) \bar{\psi}_2(s)\right|, \left|(\mathcal L_m-\mathcal L) \bar{\psi}_1(s)\right|\right\}.
\end{equation}
Applying the error estimate (\ref{errorequadest}) for the Lobatto quadrature formula, we get
\begin{eqnarray*}
\left|(\mathcal L_m-\mathcal L) \bar{\psi}_2(s)\right|&\leq &\frac{\C}{m}E_{2m}\left(\left(L(\cdot,s) \bar{\psi}_2\right)'\right)_{\varphi,1}\\
&\leq  &   \frac{\C}{m}\left[E_{2m}\left(\left(L(\cdot,s)\right)' \bar{\psi}_2\right)_{\varphi,1}
+ E_{2m}\left(L(\cdot,s) \bar{\psi}_2'\right)_{\varphi,1}  \right]=:B+C.
\end{eqnarray*}
On the other hand since, for  $f \in L^2$ and $g \in L^2_{\varphi}$, the following inequality
\begin{equation}
E_{2m}(fg)_{\varphi,1} \leq 2\|f\|_2E_m(g)_{\varphi,2}+E_m(f)_2 \|g\varphi\|_2
\end{equation}
holds true, the quantities $B$ and $C$ can be estimated as follows
\begin{eqnarray*}
B & \leq & \frac{\C}{m}\left[2\|\bar{\psi}_2\|_2E_m\left(\frac{\partial}{\partial t}L(\cdot,s)\right)_{\varphi,2}+E_m(\bar{\psi}_2)_2
\left\|\frac{\partial}{\partial t}L(\cdot,s)\varphi\right\|_2\right]\\
& =: & B_1+B_2,
\end{eqnarray*}
\begin{eqnarray*}
C & \leq & \frac{\C}{m}\left[2\|L(\cdot,s)\|_2E_m\left(\bar{\psi}_2'\right)_{\varphi,2}+E_m\left(L(\cdot,s)\right)_2\|\bar{\psi}_2'\varphi\|_2\right]\\
& =: & C_1+C_2,
\end{eqnarray*}
respectively.
Taking into account the smoothness results for the solution $\bar{\psi}$
(see (\ref{behavioursolution}), (\ref{behavioursolution2})), using the Favard inequality (\ref{Favard}) and the inequality (\ref{ineqderivL}), for any $r \in \NN$, we have
\begin{eqnarray*}
B_1 & \leq & \frac{\C}{m^r} \left\|\frac{\partial^r}{\partial t^r}L(\cdot,s)\varphi^r\right\|_{2}
\leq\frac{\C}{m^r} r!\left(\int_0^1\frac{(t+s)^{2r+2}t^r}{(t^2+2ts\cos{\chi\pi}+s^2)^{2r+2}}dt\right)^{\frac{1}{2}}\\
& = & \frac{\C}{m^r} \frac{r!}{s^{\frac{r}{2}+\frac{1}{2}}}\left(\int_0^{\frac{1}{s}}\frac{(1+x)^{2r+2}x^r}{(x^2+2x\cos{\chi\pi}+1)^{2r+2}}dx\right)^{\frac{1}{2}}
 \leq  \frac{\C}{m^r} \frac{r!}{s^{\frac{r}{2}+\frac{1}{2}}}
\end{eqnarray*}
and
\begin{eqnarray*}
B_2 & \leq & \frac{\C}{m^3} \frac{1}{s} E_{m-2}(\bar{\psi}_2'')_{\varphi^2,2}
\leq\frac{\C}{m^3} \frac{1}{s} \left\|\bar{\psi}_2''\varphi^2 \right\|_2
 \leq  \frac{\C}{m^3} \frac{1}{s}.
\end{eqnarray*}
By similar arguments we obtain
\begin{eqnarray*}
C_1 & \leq & \frac{\C}{m^2}\left(\int_0^1\left|\frac{s\sin^2{\chi\pi}}{t^2+2ts\cos{\chi\pi}+s^2)}\right|^2dt\right)^{\frac{1}{2}}E_{m-1}(\bar{\psi}_2'')_{\varphi^2,2}\\
& \leq & \frac{\C}{m^2}\frac{1}{s^{\frac{1}{2}}}\left(\int_0^{\frac{1}{s}} \frac{\sin^2{|\chi|\pi}}{(x^2+2x\cos{\chi\pi}+1)^2)}dx\right)^{\frac{1}{2}}\|\bar{\psi}_2''\varphi^2 \|_2
 \leq  \frac{\C}{m^2}\frac{1}{s^{\frac{1}{2}}}
\end{eqnarray*}
and, for any $r \in \NN$,
\begin{eqnarray*}
C_2 & \leq & \frac{\C}{m^{r+1}}E_{m-r}\left(L(\cdot,s)\right)_{\varphi^r,2}
 \leq \frac{\C}{m^{r+1}}\left\|\frac{\partial}{\partial t^r}L(\cdot,s)\varphi^r\right\|_{2}
 \leq  \frac{\C}{m^{r+1}}\frac{r!}{s^{\frac{r}{2}+\frac{1}{2}}}.
\end{eqnarray*}
Then we can deduce the pointwise estimate
\[
\left|(\mathcal L_m-\mathcal L)\bar{\psi}_2(s) \right| \leq \frac{\C}{m^2}\frac{1}{s^{\frac{1}{2}}}, \quad s \in \left[\displaystyle \frac {c}{m^{2-2\epsilon}},1\right]
\]
The same conclusion can be drawn for $\left|(\mathcal L_m-\mathcal L)\bar{\psi}_1(s)\right|$ and hence
we get
\[
\|(\tilde{\mathcal W}_m- {\mathcal W}) \bar{\psi}(s)\|_{\infty}\leq \frac{\C}{m^2}\frac{1}{s^{\frac{1}{2}}}, \quad  s \in \left[\frac {c}{m^{2-2\epsilon}},1\right].
\]
In particular when $s=\displaystyle \frac {c}{m^{2-2\epsilon}}$, one has
\[
\left\|\left(\tilde{\mathcal W}_m- {\mathcal W}\right)\bar{\psi}\left(\displaystyle \frac {c}{m^{2-2\epsilon}}\right)\right\|_{\infty}
\leq \frac{\C}{m^{1+\epsilon}}.
\]
Summing up, we can write
\[
\|(\tilde{\mathcal W}_m-{\mathcal W})\bar{\psi}(s)\|_{\infty} \leq \left\{\begin{array}{lr}
\C \max\left\{\left(\displaystyle \frac{1}{m^{2-2\epsilon}}\right)^{\beta},\displaystyle \frac{1}{m^{1+\epsilon}}\right\}, & s  \in \left[0,\displaystyle \frac {c}{m^{2-2\epsilon}}\right]\\
\displaystyle \frac{\C}{m^2}\frac{1}{s^{\frac{1}{2}}}, & s  \in \left[\displaystyle \frac {c}{m^{2-2\epsilon}},1\right]
\end{array} \right.,
\]
i.e. (\ref{errorestimate1}).
\end{proof}

\begin{proof} \emph{of Theorem \ref{theoremcondition}}
We first prove the inequality
\begin{equation} \label{ineqcond1}
\left\| A_m\right\| \leq \left\|-\pi \mathcal{I}+\tilde{\mathcal W}_m+{\mathcal S}_m\right\|.
\end{equation}
Then, take a vector
$\mathbf{a}=\left(a_{1,0},\ldots,a_{1,m+1},a_{2,0},\ldots,a_{2,m+1},a_{3,0},\ldots,a_{3,m+1}\right)^T \in \tilde{\RR}^{3(m+2)}$,
$\left\|\mathbf{a} \right\|_{\infty} \neq 0$.
For this $\mathbf{a}$, let $f=\left(f_1,f_2,f_{3}\right)^T \in \mathcal X$ such that
\[f_i(x_l)=a_{i,l}, \quad \forall i \in \{1,2,3\}, \quad \forall l \in \{0,1,\ldots,m+1\},\]
and  $\|f_i\|_{\infty}=\left\|\left(a_{i,0},\ldots,a_{i,m+1}\right)^T\right\|_{\infty}$ from which
\begin{equation} \label{normbarf}
\|f\|_{\infty}=\left\|\mathbf{a} \right\|_{\infty}.
\end{equation}
Then, taking into account (\ref{normbarf}), we can write
\begin{eqnarray*}
\left\|A_m \mathbf{a} \right\|_{\infty} & = & \max_{0 \leq l \leq m+1}\left\|(-\pi \mathcal{I}+\tilde{\mathcal W}_m+{\mathcal S}_m) f(x_l)\right\|_{\infty} \\
& \leq & \sup_{0 \leq s \leq 1} \left\|(-\pi \mathcal{I}+\tilde{\mathcal W}_m+{\mathcal S}_m)f(s)\right\|_{\infty}\\
& = & \left\|(-\pi \mathcal{I}+\tilde{\mathcal W}_m+{\mathcal S}_m )f\right\|_{\infty} \\
& \leq & \left\|-\pi \mathcal{I}+\tilde{\mathcal W}_m+{\mathcal S}_m\right\| \left\|\underline a \right\|_{\infty}
\end{eqnarray*}
from which (\ref{ineqcond1}) immediately follows. Now, in order to prove the second inequality
\begin{equation} \label{ineqcond2}
\left\| A_m^{-1}\right\| \leq \left\|(-\pi \mathcal{I}+\tilde{\mathcal W}_m+{\mathcal S}_m)^{-1}\right\|,
\end{equation}
let us to consider a vector $\mathbf{b} \in \tilde{\RR}^{3(m+2)}$ such that $\left\|\mathbf{b} \right\|_{\infty} \neq 0$ and
$\mathbf{a}=A_m^{-1} \mathbf{b}$. Pick a function $g \in \mathcal X$  with $\left\|g\right\|_{\infty}=\left\|\mathbf{b} \right\|_{\infty}$.
In correspondence of $g$ let $\varphi=\left(\varphi_1,\varphi_2,\varphi_3\right)^T$  be the array of functions defined
as $\varphi=(-\pi \mathcal{I}+\tilde{\mathcal W}_m+{\mathcal S}_m)^{-1}g$. Then (see (\ref{equivalsyst})) one has that
\[\varphi_{i}(x_l)=a_{i,l}, \quad \forall i \in \{1,2,3\}, \quad \forall l \in \{0,1,\ldots,m+1\},\]
and, hence, $\left\|\mathbf{a}\right\|_{\infty} \leq \left\|\varphi\right\|_{\infty}$.
It follows that
\begin{eqnarray*}
\left\|A_m^{-1}\mathbf{b}\right\|_{\infty} \leq \left\|\varphi\right\|_{\infty} & = &\left\|(-\pi \mathcal{I}+\tilde{\mathcal W}_m+{\mathcal S}_m)^{-1}g\right\|_{\infty}\\
& \leq & \left\|(-\pi \mathcal{I}+\tilde{\mathcal W}_m+{\mathcal S}_m)^{-1}\right\|_{\infty} \left\|\mathbf{b}\right\|_{\infty},
\end{eqnarray*}
i.e. (\ref{ineqcond2}) holds true.
Finally, combining (\ref{ineqcond1}) and (\ref{ineqcond2}) the thesis follows.
\end{proof}

\begin{proof}\emph{of Theorem \ref{harmonicerror}}
By (\ref{double2}) and (\ref{doubleappr}), for any fixed point $(x,y) \in D$,  we obtain
\begin{equation} \label{erroru}
\hspace*{0.8cm} |u(x,y)-u_m(x,y)| \leq \sum_{i=1}^{3}\left|\int_0^1 H_i(x,y,t) \bar{\psi_i}(t) dt-\sum_{h=0}^{m+1} \lambda_h H_i(x,y,x_h) \bar{\psi}_{m,i}(x_h)\right|.
\end{equation}
Then for each $i \in \{1,2,3\}$, let us estimate the $i$-th term of the previous sum as follows
\begin{eqnarray*}
\left|\int_0^1 H_i(x,y,t) \bar{\psi_i}(t) dt-\sum_{h=0}^{m+1} \lambda_h H_i(x,y,x_h) \bar{\psi}_{m,i}(x_h)\right|   \\
& & \hspace*{-7cm} \leq \left|\int_0^1 H_i(x,y,t) \bar{\psi_i}(t) dt-\sum_{h=0}^{m+1} \lambda_h H_i(x,y,x_h) \bar{\psi}_{i}(x_h)\right| \\
& & \hspace*{-7cm} +\left|\sum_{h=0}^{m+1} \lambda_h H_i(x,y,x_h) \bar{\psi}_{i}(x_h)-\sum_{h=0}^{m+1} \lambda_h H_i(x,y,x_h) \bar{\psi}_{m,i}(x_h)\right|\\
& & \hspace*{-7cm}=:A_i+B_i.
\end{eqnarray*}
By applying the error estimate (\ref{errorequad2}) for the Lobatto quadrature formula, we get
\begin{eqnarray*}
A_i & \leq & \frac {\C}{m}E_{2m}\left([H_i(x,y,\cdot)\bar{\psi}_{i}]^{'}\right)_{\varphi,1} \leq \frac {\C}{m} \left\|[H_i(x,y,\cdot)\bar{\psi}_{i}]^{'}\varphi \right\|_1\\
& \leq & \frac {\C_i}{m}\left(\frac{1}{d_i^2}+\frac{1}{d_i}\right)
 \leq  \frac {\C_i}{m}\left(\frac{1}{d^2}+\frac{1}{d}\right)
\end{eqnarray*}
where $d_i=\displaystyle \min_{0 \leq t \leq 1}|(x,y)-(\xi_i(t),\eta_i(t))|$,  $d=\displaystyle \min_{i=1,2,3}d_i$ and $\C_i \neq \C_i(x,y)$.
Now, for the quantity $B_i$ we can write
\begin{eqnarray*}
B_i & \leq & \sum_{h=0}^{m+1} \lambda_h \left|H_i(x,y,x_h)\right| \left|\bar{\psi}_{i}(x_h)-\bar{\psi}_{m,i}(x_h)\right| \\ & \leq & \left\|\bar{\psi}_{i}-
\bar{\psi}_{m,i}\right\|_{\infty}\left\|H_i(x,y,\cdot)\right\|_{\infty}\sum_{h=0}^{m+1} \lambda_h \leq \frac{\C_i'}{d}\left\|\bar{\psi}-\bar{\psi}_m\right\|_{\infty}
\end{eqnarray*}
with the constant $\C_i'$ independent of $(x,y)$.
Hence, combining (\ref{erroru}) with the previous estimates for $A_i$ and $B_i$ we can deduce the thesis.
\end{proof}

\section{Numerical examples} \label{Section5}
In this section we consider some examples of the interior Dirichlet problem defined on planar domains with a  corner and solve them
by means of the numerical method proposed in Section \ref{Section3}.
In order to give the boundary condition $g$, we choose  a test harmonic function $u$. After solving the linear system (\ref{linearsystem}),
we compute the approximate array $\bar{\psi}_m$, solution of (\ref{sistapprox}), and the function $u_m$, defined in (\ref{doubleappr}),
which approximates the double layer potential $u$. \newline
In the following tables we perform a discrete version of $\left\|\bar{\psi}_m  \right\|_{\infty}$, (reporting only the digits which are correct
according to the value obtained for $m=2048$), the absolute error
 $\varepsilon_m(x,y)=|u(x,y)-u_m(x,y))|$ in different points $(x,y) \in D$
and the condition numbers in infinity norm of the matrix $A_m$ of the system (\ref{linearsystem}).\\ \\
{\bf Example 1.} 
Consider the Dirichlet problem (\ref{Dirichlet}) on a domain having a reentrant corner $P_0=(0,0)$ with interior angle
$\phi=\frac 32 \pi$ and a contour  $\Sigma$ given by the following parametric representation
\[\sigma(t)=\left(\frac 23 \sin{(3 \pi t)}, \sin{(2 \pi t)}\right), \quad  t \in [0,1],\]
(see Figure \ref{boundary1}).
\begin{figure}[t]
\centering
\includegraphics[width=0.65\textwidth]{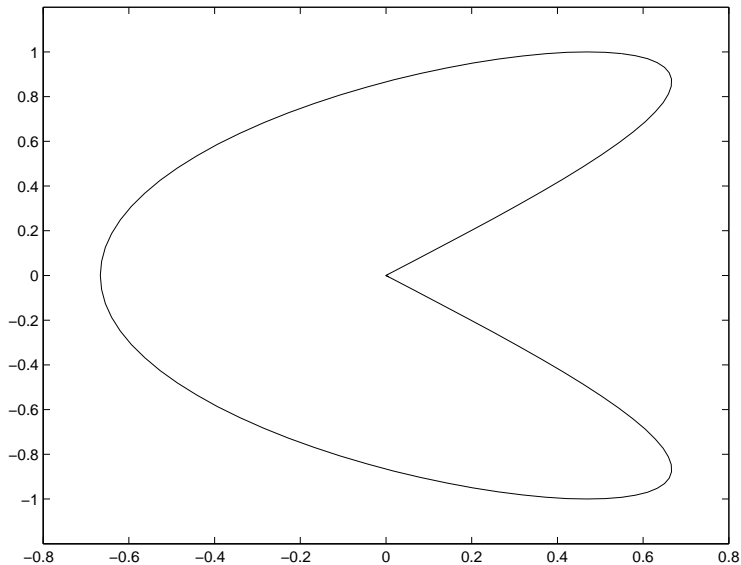}
\caption{The contour $\Sigma$ in Example 1}
\label{boundary1}
\end{figure}
Moreover, we assume that the solution of (\ref{Dirichlet}) is
the harmonic function
\[u(x,y)=r^{\frac 23} \cos{\frac{2}{3}\theta}, \]
in polar coordinates $r,\theta$, to give a realistic behaviour of $u$ at the corner (see \cite{CostSte,Gris}).
Then the boundary datum $g$  is given by setting $g=u$ on $\Sigma$.\newline
By applying our numerical procedure, we have chosen the length of the two sections $\Sigma_1$ and $\Sigma_2$ intersecting
at the corner point such that $\delta=5.16e-08$ and the parameters involved in the definition (\ref{opmodified}) of the modified operator
$\tilde{\mathcal W}_m$ given by $c=50$, $\epsilon=10^{-3}$.

Tables \ref{table1} and \ref{table2} give the numerical results. They show that the linear system we solve is well conditioned for each sufficiently large
value of $m$, the sequence of the approximating arrays $\bar{\psi}_m$ converges and, also, that the error in the approximation of the double layer potential
becomes smaller and smaller as well as we move away from the boundary.

\begin{table}[htb]
\centering
\caption{Condition numbers and  norm of $\bar{\psi}_m$}
{
\begin{tabular}{lll}
\hline\noalign{\smallskip}
$m$  & \quad $\mathrm{cond}(A_m)$ & \quad $\|\bar{\psi}_m\|_{\infty}$ \\
\noalign{\smallskip}\hline\noalign{\smallskip}
64  &  \quad 16.92 &   \quad 2.33052e-01\\
128 &  \quad 16.93 &   \quad 2.33052e-01\\
256 &  \quad 16.93 &   \quad 2.330523e-01 \\
512 &  \quad 16.93 &   \quad 2.330523e-01 \\
\noalign{\smallskip}\hline
\end{tabular}
}
\label{table1}
\end{table}

\begin{table}[htb]
\centering
\caption{Errors $\varepsilon_m(x,y)$}
{
\begin{tabular}{lllll}
\hline\noalign{\smallskip}
$m$   & $\varepsilon_m(-0.01,0)$ & $\varepsilon_m(0,0.1)$ & $\varepsilon_m(-0.4,0.4)$ &$\varepsilon_m(0.4,0.8)$ \\
\noalign{\smallskip}\hline\noalign{\smallskip}
64  &    7.61e-05  &  4.74e-06  &  8.53e-04 &   7.54e-06   \\
128 &    7.02e-06  &  9.41e-07  &  1.46e-05 &   1.12e-08   \\
256 &    1.38e-06  &  1.96e-07  &  3.43e-08 &   2.34e-09   \\
512 &    7.21e-08  &  4.77e-09  &  1.87e-09 &   1.03e-10   \\
\noalign{\smallskip}\hline
\end{tabular}
}
\label{table2}
\end{table}

{\bf Example 2.} 
Consider the Dirichlet problem (\ref{Dirichlet}) on a drop-shaped domain having a corner point $P_0=(0,0)$ with interior angle $\phi=\frac{2}{3} \pi$ whose contour $\Sigma$
is represented by the following parametrization
\[\sigma(t)=\left(\frac{2}{\sqrt{3}} \sin{\pi t} , -\sin{2 \pi t}\right) \quad t \in [0,1], \]
(see Figure \ref{boundary2}).
\begin{figure}[t]
\centering
\includegraphics[width=0.65\textwidth]{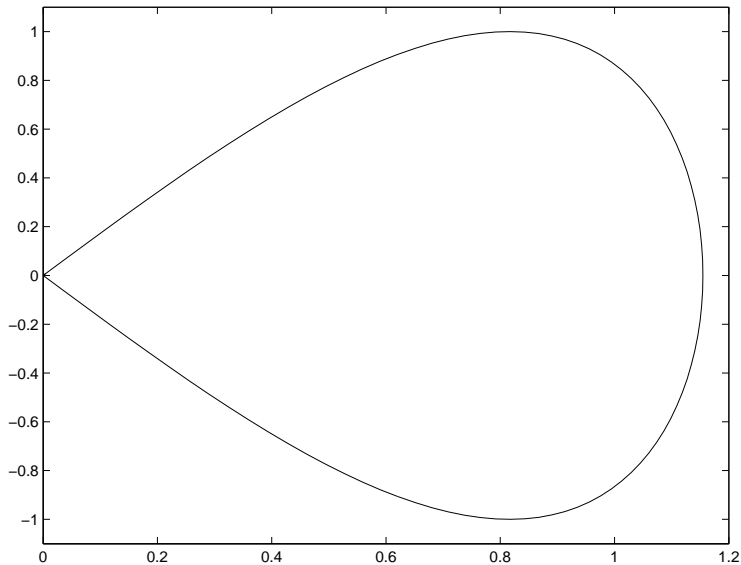}
\caption{The contour $\Sigma$ in Example 2}
\label{boundary2}
\end{figure}
 The boundary data $g$ is given through the harmonic function
\[u(x,y)=r^{\frac 32} \cos{\frac{3}{2}\theta},\]
in polar coordinates $r,\theta$, chosen because of its realistic behavior near the corner.
In this case we have chosen $\delta=3.10e-08$ and the parameters in (\ref{opmodified})
as follows: $c=1$, $\epsilon=10^{-6}$.
In Table \ref{table3} and Table \ref{table4} we have reported the numerical results. Let us observe that one can repeat word by word the
comments made in the previous example.

\begin{table}
\centering
\caption{Condition numbers and  norm of $\bar{\psi}_m$}
{
\begin{tabular}{lll}
\hline\noalign{\smallskip}
$m$  & $\mathrm{cond}(A_m)$ & $\|\bar{\psi}_m\|_{\infty}$\\
\noalign{\smallskip}\hline\noalign{\smallskip}
64  &   4.80  &  4.4387 e-01      \\       
128 &   4.20  &  4.438746e-01    \\        
256 &   4.18  &  4.43874669e-01  \\        
512 &   4.18  &  4.438746696045e-01   \\       
\noalign{\smallskip}\hline
\end{tabular}
}
\label{table3}
\end{table}

\begin{table}[htb]
\centering
\caption{Errors $\varepsilon_m(x,y)$}
{
\begin{tabular}{lllll}
\hline\noalign{\smallskip}
$m$  & $\varepsilon_m(0.01,0)$ & $\varepsilon_m(0.1,0)$ & $\varepsilon_m(0.8,0.6)$ &$\varepsilon_m(0.9,0.8)$\\
\noalign{\smallskip}\hline\noalign{\smallskip}
64  & 8.78e-03 &  6.59e-05 &  6.84e-07 &  4.78e-05 \\
128 & 6.66e-05 &  1.06e-06 &  8.47e-09 &  1.59e-09 \\
256 & 5.24e-08 &  1.72e-09 &  1.37e-11 &  3.08e-12 \\
512 & 5.84e-11 &  1.84e-12 &  1.25e-14 &  8.77e-15 \\
\noalign{\smallskip}\hline
\end{tabular}
}
\label{table4}
\end{table}

{\bf Example 3.} 
We test our method for the domain whose boundary $\Sigma$ admits the follo\-wing parametric representation:
\[\sigma(t)=\sin{\pi t}\left(\cos{((1-\chi)\pi t)},\sin{((1-\chi)\pi t)} \right) \quad t \in [0,1], \quad \chi=0.86 \]
with a single corner at $P_0=(0,0)$ (see Figure \ref{boundary3}). The interior angle at $P_0$ is $\phi=(1-\chi)\pi$.
\begin{figure}[t]
\centering
\includegraphics[width=0.65\textwidth]{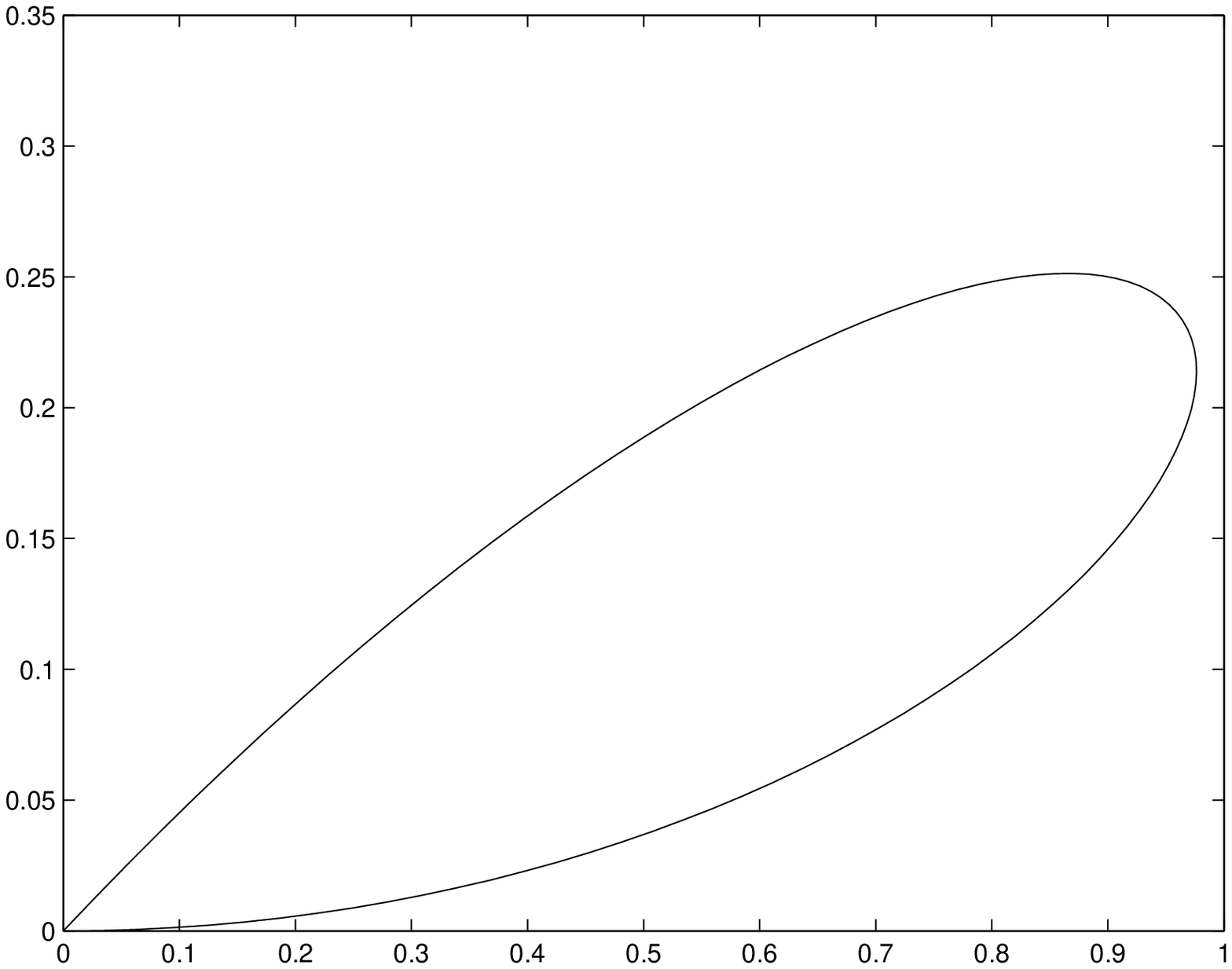}
\caption{The contour $\Sigma$ in Example 3}
\label{boundary3}
\end{figure}
Here we have chosen as exact solution the harmonic function
$$u(x,y)=\sin{x}\cosh{y}$$
and $\delta=1.52e-08$, $c=500$ and $\epsilon=10^{-1}$.
The numerical results are shown in tables \ref{table5} and \ref{table6}.
\begin{table}[htb]
\centering
\caption{Condition numbers and  norm of $\bar{\psi}_m$}
{
\begin{tabular}{lll}
\hline\noalign{\smallskip}
$m$  & $\mathrm{cond}(A_m)$ & $\|\bar{\psi}_m\|_{\infty}$ \\
\noalign{\smallskip}\hline\noalign{\smallskip}
64  &  49.38  &  1.4456e-01  \\
128 &  58.86  &  1.4456e-01 \\ 
256 &  15.05  &  1.44568e-01\\
512 &  14.06  &  1.44568490e-01 \\ 
1024&  14.12  &  1.445684902e-001 \\ 
\noalign{\smallskip}\hline
\end{tabular}
}
\label{table5}
\end{table}

\begin{table}
\centering
\caption{Errors $\varepsilon_m(x,y)$}
{
\begin{tabular}{lllll}
\hline\noalign{\smallskip}
$m$  & $\varepsilon_m(0.05,0.01)$ & $\varepsilon_m(0.2,0.025)$ & $\varepsilon_m(0.4,0.05)$ &$\varepsilon_m(0.8,0.15)$ \\
\noalign{\smallskip}\hline\noalign{\smallskip}
64 &   4.46e-04 &    7.32e-03 &   3.12e-03 &    4.27e-04 \\
128 &  1.45e-05 &    2.70e-04 &  3.44e-04  &  1.61e-06\\
256 &  9.62e-08 &   2.76e-07 &   9.19e-08  &  1.16e-12\\
512 &  9.04e-14 &   1.34e-13   &   2.44e-13  &  1.99e-15 \\
1024&  0 &   6.66e-16  &  2.83e-15 &    3.66e-15 \\
\noalign{\smallskip}\hline
\end{tabular}
}
\label{table6}
\end{table}

{\bf Example 4.} 
In this example, in order to focus our attention on the behavior of the condition number $\mathrm{cond}(A_m)$
when the interior angle varies,  we consider a family of domains bounded  by the curves
\[\sigma(t,\phi)=\sin{\pi t}\left(\cos{\phi\left(t-\frac 1 2\right)},\sin{\phi\left(t-\frac 1 2\right)} \right) \quad t \in [0,1], \]
with a corner at $P_0=(0,0)$ and interior angles $\phi \in [0.1\pi,1.9\pi]$.
Figure \ref{plotcond} shows, for some fixed (and sufficiently large) values of $m$, the plot of $\mathrm{cond}(A_m)$ as a function of the interior angle $\phi$. The graphs were obtained in correspondence of the following choice of the parameters involved in the numerical procedure: $c=200$,
$\epsilon=10^{-1}$. They confirm our theoretical expectations. In fact we can note that, for a fixed $m$, the condition numbers of the matrix $A_m$ are
small for each value of $\phi$. On the other hand, they put in evidence that
the sequence $\left\{\mathrm{cond}(A_m)\right\}_{m \geq m_0}$ is uniformly bounded with respect to $m$,
according with estimate (\ref{conditionnumber}).

\begin{figure}[t]
\centering
\includegraphics[width=0.65\textwidth]{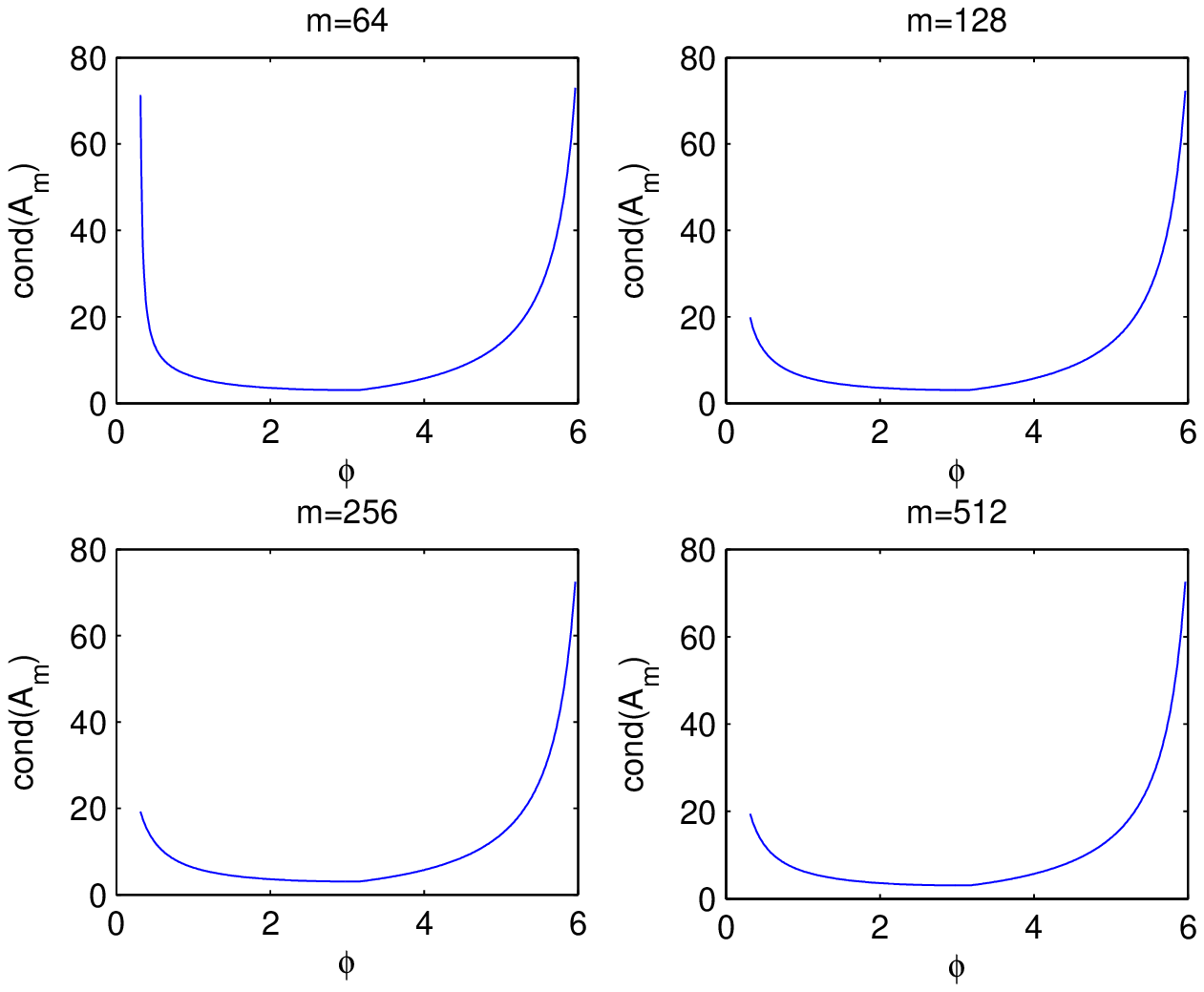}
\caption{Condition numbers for Example 4}
\label{plotcond}
\end{figure}

{\bf Example 5.} 
We can repeat word by word the remarks of the previous example when we consider the family of
 ``heart-shaped" domains bounded by the curves
$$\sigma(t)=\left(\begin{array}{c}\cos{(1+\frac{\phi}{\pi})\pi t}-\sin{(1+\frac{\phi}{\pi})\pi t}\\
\sin{(1+\frac{\phi}{\pi})\pi t}+\cos{(1+\frac{\phi}{\pi})\pi t}\end{array}\right)\left(\begin{array}{c}\tan{\frac{\phi}{2}}\\ 1\end{array}\right)
-\left(\begin{array}{c} \tan{\frac{\phi}{2}}\\ \cos{\pi t}\end{array}\right), \quad t \in [0,1], $$
with $\phi \in (\pi,2\pi)$ the interior angle of the single outward-pointing corner $P_0=(0,0)$. The behavior of the condition numbers $\mathrm{cond}(A_m)$ is illustrated by  Figure \ref{plotcond2}.
\begin{figure}[t] \label{plotcond2}
\centering
\includegraphics[width=0.65\textwidth]{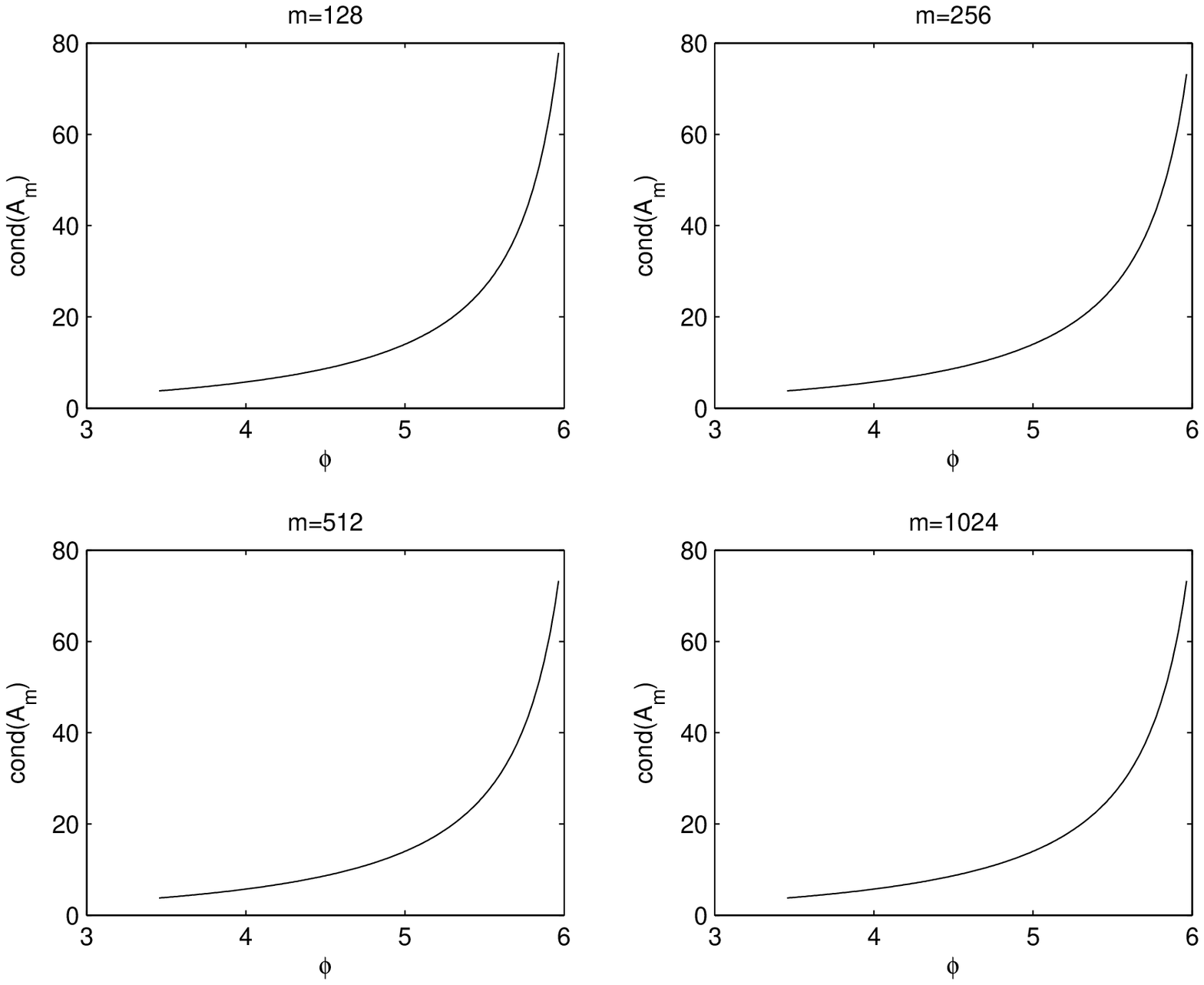}
\caption{Condition numbers for Example 5}
\label{plotcond2}
\end{figure}

\section*{Acknowledgments}
C. Laurita is partly supported by GNCS Project 2013 ``Metodi fast per la risoluzione numerica di sistemi di equazioni integro-differenziali''.

\bibliographystyle{plain}
\bibliography{refs}
\end{document}